\documentclass[letterpaper,10pt]{article}
\usepackage[colorlinks=true,linkcolor=blue,citecolor=red]{hyperref} 

\usepackage{amsmath}

\usepackage{float}
\usepackage{amssymb}
\usepackage{amsfonts}
\usepackage{graphicx}
\usepackage{epstopdf}
\usepackage{color}
\usepackage{multirow}
\usepackage[title]{appendix}
\usepackage{arydshln} 

\graphicspath{{figures/}}
\usepackage{subcaption}

\usepackage{xcolor, colortbl}
\definecolor{Gray}{gray}{0.90}
\definecolor{LightCyan}{rgb}{0.65,1,1}
\definecolor{Cyan}{rgb}{0.95,1,1}
\definecolor{GRAY}{gray}{0.75}

\usepackage[T1]{fontenc}
\usepackage[utf8]{inputenc}
\usepackage[font=small,labelfont=bf]{caption}

\usepackage[numbers,sort&compress]{natbib}

\newcommand{\bsub}{\begin{subequations}}
\newcommand{\esub}{\end{subequations}$\!$}

\numberwithin{equation}{section}

\usepackage[numbers,sort&compress]{natbib}

\DeclareTextFontCommand{\textcyr}{\cyr}

\usepackage{tabularx}
\usepackage{dirtytalk}
\usepackage{amsthm}
\usepackage{wrapfig}
\usepackage{mathtools}
\usepackage{bbm}

\usepackage[format=plain,
            labelfont={it},
            textfont=it]{caption}

\numberwithin{equation}{section}
\newtheorem{theorem}{Theorem}[section]

\newtheorem{lemma}[theorem]{Lemma}

\newtheorem{cor}[theorem]{Corollary}
\newtheorem{assumption}[theorem]{Assumption}
\newtheorem{remark}[theorem]{Remark}

\newcommand{\bbR}{\mathbb{R}}
\newcommand{\R}{\mathbb{R}}

\newcommand{\bbZ}{\mathbb{Z}}
\newcommand{\Z}{\mathbb{Z}}
\newcommand{\bbC}{\mathbb{C}}
\newcommand{\cO}{\mathcal{O}}
\newcommand{\cL}{\mathcal{L}}

\newcommand{\range}{\mathrm{range}}
\newcommand{\dm}{\mathbbm{d}}
\renewcommand{\Re}{\,\mathrm{Re}\,}
\renewcommand{\Im}{\,\mathrm{Im}\,}

\newcommand{\vx}{{\bf{x}}}

\newcommand{\vn}{{\bf{n}}}

\begin{document}
\begin{center}
{\fontsize{14}{14}\fontfamily{cmr}\fontseries{b}\selectfont{
Oscillations in a scalar differential equation coupled to a diffusive field}}\\[0.2in]
Merlin Pelz$^{\dagger}$ and  Arnd Scheel$^{\dagger}$\\[0.1in]
\textit{\footnotesize
$^\dagger$School of Mathematics, University of Minnesota, Minneapolis, 206 Church St SE, MN 55455, USA}
\end{center}

\begin{abstract}
We study the emergence of periodic oscillations through a Hopf bifurcation in a scalar diffusion equation on the half line coupled to a dynamic boundary condition. Our results quantify the effect of delay through the buffering in the diffusive field on boundary kinetics, drawing a parallel to the emergence of oscillations in delay equations. Technically, the Hopf bifurcation occurs in the presence of essential spectrum induced by the diffusive field, preventing a simple approach via center-manifold reduction. The results are motivated by observations in biological systems where dynamic boundary conditions arise when modeling surface dynamics coupled to bulk diffusion.
\end{abstract}

\section{Introduction}\label{sec:intro}

We study a simple scalar diffusion equation for $u(t,x)\in\R$, $t\in\R$, in the half-space $x>0$, coupled to a dynamic (or Wentzell) boundary condition,
\begin{subequations} \label{eq:sys}
\begin{eqnarray}
    \partial_t u &=& \partial_{xx}u - \sigma^2 u, \quad x\in(0,\infty),\label{eq:diff}\\
    u(t,0)&=&u^-(t),\label{eq:bc}\\
    \frac{d}{dt} u^- &=& f(u^-,\partial_n u(t,0),\mu).\label{eq:dynb}
\end{eqnarray}
\end{subequations}

Here, $u$ is a scalar diffusive field with degradation $\sigma^2$,~\eqref{eq:diff}, $u^-$ is the Dirichlet boundary value,~\eqref{eq:bc}, which evolves according to a scalar ordinary differential equation with smooth vector field $f$ that couples to the diffusive field through its dependence on the outer normal derivative $\partial_n u=-\partial_xu$ at $x=0$,~\eqref{eq:dynb}.   Systems of the form~\eqref{eq:sys} arise in many circumstances where boundary or surface dynamics represented by $u^-$, interact with diffusive fields, represented by $u$, possibly with degradation $\sigma^2\geq0$. Somewhat recently, the possibility of oscillations induced by the buffer effect of the diffusive field has received renewed attention~\cite{scheel2021signaling,brauns2021bulk}.
Our main result demonstrates that this simple equation naturally gives rise to Hopf bifurcations and, depending on the direction of bifurcation, stable periodic orbits. Our results are novel in at least two aspects.  First, the presence of Hopf bifurcations is remarkable since scalar reaction-diffusion equations with separated boundary conditions, even with gradient dependence, do not possess periodic solutions but rather strict Lyapunov functions~\cite{zelenjak,LF}. Second, more technically, our approach resolves difficulties due to the presence of essential spectrum in a linearized version of~\eqref{eq:sys} caused by the diffusive field when $\sigma=0$. We in fact quantify the effect of oscillations on the behavior $u(t,x)$ in the far-field, that is, for $x\to\infty$, continuously in $\sigma\sim 0$, allowing for constant limits as $\sigma =0$ and weak exponential decay or even growth when $\sigma\sim 0$. 

In the remainder of this introduction, we first state our main result and then comment on the novel aspects, including a review of the dynamic boundary conditions, the interpretation of buffering as a delay term, and other results and treatments of Hopf bifurcation in the presence of essential spectrum.

\paragraph{Main result --- assumptions.}

The assumptions set up a situation where  a branch of steady-state solutions $\{u_\ast^-(\mu, \sigma),u_\ast(x;\mu, \sigma)\}$ of~\eqref{eq:sys} destabilizes at $(\mu, \sigma) = (0, \sigma_*)$ due to oscillations localized at the boundary. Note that  for steady-states, $u_\ast(x) = e^{-\sigma x} u_\ast^-$, so that $\partial_n u_\ast(t, 0) = \sigma u_\ast^-$.

\begin{assumption}[Existence of equilibrium]\label{assu:equi}
    We assume  that $f(u_\ast^-, \sigma_\ast u_\ast^-,0)=0$ so that $u^-=u_\ast^-$ and $u_\ast(x)=u_\ast^- e^{-\sigma_\ast x}$ is an equilibrium solution to~\eqref{eq:sys}.
    
\end{assumption}
In order to analyze the stability of  $u_\ast$, we formally linearize~\eqref{eq:sys} at $u_\ast$, set $\mu=0$, $\sigma=\sigma_\ast$, and look for solutions $v^-(t)=e^{\lambda t}v^-_0$, $v(t,x)=e^{\lambda t} v_0(x)$, to find
\begin{subequations} \label{eq:syslin}
    \begin{eqnarray}
       \lambda {v}^-_0 &=& \partial_1f(u_\ast^-, \sigma_\ast u_\ast^-, 0)v^-_0 + \partial_2 f(u_\ast^-, \sigma_\ast u_\ast^-, 0) \; \partial_n v_0(t,0),  \label{eq:dynblin}  \\
        \lambda v_0 &=& \partial_{xx} v_0 - \sigma_\ast^2 v_0, 
        \label{eq:difflin} \\
        v_0( 0) &=& v^-_0.\label{eq:bclin}
    \end{eqnarray}
\end{subequations}
Here the partial derivatives $\partial_j$ are understood as derivatives with respect to the $j^\text{th}$ argument of the function $f$. Bounded solutions to~\eqref{eq:difflin}--\eqref{eq:bclin}  are of the form $v_0(x)=e^{-\sqrt{\lambda + \sigma_\ast^2} \; x}\;v^-_0$ for $\lambda\in \bbC\setminus(-\infty,- \sigma_\ast^2]$, where we take the square root with branch cut along $(-\infty,0]$ so that $\Re\sqrt{\lambda+ \sigma_\ast^2}\geq 0$.
Substituting this explicit form into~\eqref{eq:dynblin} gives nontrivial solutions precisely when 
\begin{equation}\label{eq:char}
    \dm_\ast(\lambda) := \lambda - \partial_1 f(u_\ast^-, \sigma_\ast u_\ast^-, 0) - \partial_2 f(u_\ast^-, \sigma_\ast u_\ast^-, 0) \sqrt{\lambda + \sigma_\ast^2} = 0.
\end{equation}
We shall assume that there exists a marginally stable oscillation. 

\begin{assumption}[Simple imaginary eigenvalues]\label{assu:degen}
    We assume 
    \begin{equation}\label{e:asspt12}
        (1) \; \exists \,\omega_\ast > 0: \; \dm_\ast(i\omega_\ast) = 0, \quad (2) \;\forall \, \omega \in \bbR \;\text{with}\ |\omega| \;\neq \omega_\ast: \; \dm_\ast(i\omega) \neq 0, \quad (3) \; \partial_\lambda \dm_\ast(\lambda) \Big|_{\lambda = i\omega_\ast} \neq 0.
    \end{equation}
    Here,  (1) implies the existence of oscillations with  neutral growth and frequency $\omega_\ast$, (2) implies the absence of solutions with different frequency but also neutral growth rate $\Re(\lambda) = 0$, and (3) implies that the eigenmode with growth rate $\lambda = i\omega_\ast$ is algebraically simple. 
\end{assumption}
From~\eqref{e:asspt12},\,(2), $\dm_\ast(0)\neq 0$, so that $\partial_1 f(u_\ast^-(0, \sigma_\ast), \sigma_\ast u_\ast^-(0, \sigma_\ast), 0) + \partial_2 f(u_\ast^-(0, \sigma_\ast), \sigma_\ast u_\ast^-(0, \sigma_\ast), 0) \;\sigma_\ast\neq 0$. The Implicit Function Theorem, then  allows us to track the steady-state as a function of $\mu$ and $\sigma$ in a neighborhood of $\mu=0$ and $\sigma = \sigma_\ast$ with  resulting smooth branch of solutions,
$u_\ast^-(\mu, \sigma)$,
\begin{equation}\label{e:eqb}
f(u_\ast^-(\mu, \sigma),\sigma u_\ast^-(\mu, \sigma),\mu)=0,
\end{equation} 
slightly abusing notation, with $u_\ast^-(0, \sigma_\ast)=u_\ast^-$. Along this branch, we can linearize and derive a characteristic equation analogous to~\eqref{eq:char}, which we  denote by $\dm$, 
\begin{equation}\label{eq:dlambdaIFT}
    \dm(\lambda; \mu, \sigma) := \lambda - \partial_1 f(u_\ast^-(\mu, \sigma), \sigma u_\ast^-(\mu, \sigma), \mu) - \partial_2 f(u_\ast^-(\mu, \sigma), \sigma u_\ast^-(\mu, \sigma), \mu) \sqrt{\lambda + \sigma^2}.
\end{equation}
By Assumption~\ref{assu:degen}, $\dm(i\omega_\ast;0, \sigma_\ast)=0$ and $\partial_\lambda \dm(i\omega_\ast;0, \sigma_\ast)\neq0$, so that, again with the Implicit Function Theorem, there is a unique family of roots $\lambda_\ast(\mu, \sigma)$ for $\mu\sim 0$ and $\sigma\sim \sigma_\ast$ with
\[
\dm(\lambda_\ast(\mu, \sigma);\mu, \sigma)=0,\qquad \lambda_\ast(0, \sigma_\ast)=i\omega_\ast.
\]
\begin{assumption}[Strict crossing]\label{assu:bif}
    We assume that 
    \begin{equation} \label{eq:strictcrossing}
        \Re(\partial_\mu\lambda_\ast(0, \sigma_\ast)) \neq 0.
    \end{equation}
    In words, the exponential  growth rate $\lambda$  of perturbations is increasing or decreasing with nonzero speed as we cross the critical parameter value $\mu=0$ at fixed $\sigma = \sigma_\ast$. 
\end{assumption} 
With~\eqref{eq:strictcrossing}, we can solve $\Re(\lambda_\ast(\mu, \sigma)) = 0$ for $\mu=\mu_\ast(\sigma)$ and define $\omega_\ast(\sigma) = \Im(\lambda_\ast(\mu_\ast(\sigma), \sigma))$ close to $\sigma = \sigma_\ast$.

\paragraph{Main result --- statement.}

With these assumptions, we are now ready to state our first main result on the existence of nonlinear, time-periodic solutions bifurcating from the branch of spatially constant solutions in~\eqref{eq:sys}. Using~\eqref{eq:strictcrossing}, we can find nearby critical crossings,  $\Re(\lambda_\ast(\mu, \sigma)) = 0$ for $\mu=\mu_\ast(\sigma)$, with frequencies $\omega_\ast(\sigma) = \Im(\lambda_\ast(\mu_\ast(\sigma), \sigma))$.

\begin{theorem} \label{theorem:hopf}
    Let Assumption~\ref{assu:equi},~\ref{assu:degen}, and~\ref{assu:bif} be satisfied. Then for sufficiently small amplitudes $|r|<r_0$ and sufficiently small $|\sigma - \sigma_\ast|$, there exist smooth functions $\mu_\mathrm{nl}$ and $\omega_\mathrm{nl}$ of $r$ and $\sigma$, even in $r$,  with
        \begin{equation}
            \mu_{\mathrm{nl}}(0, \sigma) = \mu_\ast(\sigma),\qquad
            \omega_{\mathrm{nl}}(0, \sigma) = \omega_\ast(\sigma),
        \end{equation}
        and, with time-reparameterization $s = \omega_{\mathrm{nl}} t$, periodic functions
        \begin{equation}
            \tilde{u}(s, x; r, \sigma) = \
            \tilde{u}(s+2\pi, x; r, \sigma),
        \end{equation}
        such that $u(t,x; r, \sigma)=\tilde{u}(\omega_{\mathrm{nl}}(r, \sigma)t, x; r, \sigma)$ has boundary values  $u^-(t; r, \sigma) := \tilde{u}(\omega_{\mathrm{nl}}(r, \sigma)t, 0; r, \sigma)$ and is a solution to~\eqref{eq:sys} with $\mu=\mu_\mathrm{nl}(r, \sigma)$. Moreover, there exist $C, \eta > 0$ and $u_\infty(r, \sigma) \in \bbR$ such that
        \begin{eqnarray}\label{e:asy}
            |\tilde{u}(s, x; r, \sigma) - u_\infty(r, \sigma) e^{-\sigma x}| &\leq& C e^{-\eta x} \qquad \qquad \text{as} \quad x\to \infty,
        \end{eqnarray}
        for all $(r, \sigma)$ close to $(0, \sigma_\ast)$.
        The amplitude of oscillations at the boundary is given by $r$, that is, 
        \begin{equation}
            \tilde{u}(s,0; r, \sigma) = r\cos(s) + \cO(r^2).
        \end{equation}
        The branch of solutions is unique up to shifts in $s$ when assuming boundedness in the sense of~\eqref{e:asy}.
\end{theorem}

\begin{remark}[Far-field asymptotics]\label{r:ff}
    Implications of the asymptotics~\eqref{e:asy} are of course only nontrivial when $\sigma\sim 0$. Particularly, when $\sigma=0$, they quantify an asymptotic limit of the oscillations that is possibly \emph{not} the equilibrium $u_\ast^-(\mu,\sigma)$. Since the diffusion equation preserves the limit $\lim_{x\to\infty}u(t,x)$ in time, we then expect that a localized perturbation of the equilibrium $u_\ast^-(\mu,\sigma)$ approaches the periodic solution only locally uniformly, inducing an error-function type spreading in the far-field. On the other hand, the theorem also yields solutions with ``pure'' exponential growth $e^{-\sigma x}$ when $\sigma \lesssim 0$. Here, pure exponential growth refers to the absence of a weakly decaying term $e^{\sigma x}$ in the asymptotics, a condition that is well known from the analysis of resonances in Schr\"odinger operators.
\end{remark}

\paragraph{Dynamic or Wentzell boundary conditions.} Clearly, the presence of oscillations is intimately related to the presence of a dynamic boundary condition, as opposed to, say, a nonlinear Robin condition, which once linearized would yield only real eigenvalues of the linearization. The boundary conditions were derived by  Wentzell (sometimes spelled Venttsel) in 1959  in a closed bounded region as a most general boundary condition  to an elliptic operator $\cL$ representing an infinitesimal generator of a Markov process in the region. The boundary condition involves terms that are not computable only with values along the boundary and therefore, as Wentzell pointed out, is not a boundary condition in the strict sense of the term. A more recent derivation can be found in~\cite{goldstein2006derivation}, with Wentzell boundary conditions for the diffusion and wave equations resulting in, for the linear case in arbitrary dimension $N$,  a boundary differential equation $\frac{d}{dt} u(t, \vx_0) = -a(\vx_0)u(t, \vx_0) + b(\vx_0)\; \partial_\vn u(t, \vx_0)$ for $\vx_0 \in \partial \Omega$ of a domain $\Omega \subset \bbR^N$.

We here refer to the boundary condition~\eqref{eq:dynb} as a dynamic boundary condition, since it involves the partial time derivative, which cannot easily be eliminated replacing it by $\Delta u(t, 0) - \sigma^2 u^-$ since $\Delta u(t, 0)$ on the boundary point is not defined when, say, posing the equation for $u(0,\cdot)\in L^2$. 

A simple interpretation of the boundary condition can be gained when thinking of a discretized version of the diffusion equation, with small $\cO(dx)$-sized compartments in the bulk, $j \in \{1,2,\ldots\}$, and an order-one sized compartment at the boundary, $j=0$, each compartment well-mixed with constant concentrations $u_j$; see also~\cite{scheel2021signaling}. Assuming fluxes of strength $1/dx$, one obtains equations for the concentrations at $u_j$, $j \in \{1,2,\ldots\}$, in the bulk with simple 3-point stencil for the Laplacian $\frac{d}{dt}u_j=(u_{j+1}+u_{j-1} -2 u_j)/dx^2$, $j \in \{1,2,\ldots\}$, and $\frac{d}{dt}u_0=(u_1-u_0)/dx$. In the limit $dx\to 0$, this gives
\[
\partial_tu=\partial_{xx}u,\ x>0,\qquad u|_{x=0}=u^-,\qquad \frac{d}{dt}u^-=-\partial_{n}u|_{x=0},
\]
with conserved total mass $u^-+\int u$. In our more general system~\eqref{eq:sys}, we also assume that the order-one size compartment $u_0$ exhibits a nonlinear reaction, modeled through the dependence of $f$ on $u^-$, which in addition depends on the flux $\partial_n$ in a nontrivial fashion. In fact, we shall see in Section~\ref{sec:examples_nodegrad} that key to the presence of a Hopf bifurcation is a positive coefficient in the dependence on the flux $\partial_n u$, for instance  $\frac{d}{dt}u^-=-\alpha u +\sqrt{2\alpha}\,\partial_{n}u|_{x=0}$. Such terms were motivated in~\cite{scheel2021signaling} through directed motion and resulting gradient sensing of stem cells.

\paragraph{Gradient sensing and polarity preservation in planarian regeneration.}
A scalar system of the form~\eqref{eq:sys} was derived in~\cite{scheel2021signaling} to model the robust regeneration of flatworms after cutting and grafting experiments. 
When cut into pieces as small as 0.5\% of the original size, each of the pieces regenerates into a fully functional planarian~\cite{rink2018stem},  while maintaining its original orientation, that is, it preserves polarity so that the head regenerates from the body part that was situated closest to the head before injury.   The model in~\cite{scheel2021signaling} focuses on the 1-D anterior-posterior axis of the planarian, incorporates dynamic boundary conditions as in~\eqref{eq:sys}, and can in fact robustly preserve polarity and reproduce behavior closely resembling lab experiments.

The model derived in~\cite{scheel2021signaling} is based on experimental evidence of a global signal gradient, which then even in small cut out body parts of otherwise roughly  homogeneous tissue preserves directional information. Regeneration via stem cell migration and differentiation is then guided by this signal gradient to induce differentiation in a somewhat separated wound healing region based on the direction of the gradient relative to the wound surface. 

The basic model in~\cite{scheel2021signaling} involves dynamic variables for concentrations of stem cells, head cells, tail cells, a signal magnitude guiding stem cell differentiation into head cells, another signal magnitude guiding stem cell differentiation into tail cells, and the magnitude of a long-range wnt-related signal. Wnt is a family of lipid-modified signaling glycoproteins~\cite{cadigan1997wnt}. Their signaling pathways use either nearby cell-cell communication (paracrine) or same-cell communication (autocrine) and are highly evolutionarily conserved in animals and humans~\cite{nusse1992wnt}. 

Using an adiabatic reduction for the kinetics by equilibrating  fast dynamics of the signal magnitudes leading to stem cell differentiation and using a nonlinear analysis based on front propagation,  a reduced model involving only the magnitude of the wnt-related signal and an order parameter that incorporates both the concentrations of head cells and tail cells is derived. In a final simplification the authors equilibrate the order parameter itself and arrive at the following reduced system involving only the magnitude $w$ of the wnt-related signal, where the key feature is the restoration of the signal gradient. The system derived there is a simple scalar diffusive field with a dynamic boundary condition  of the form 
\begin{equation}
\partial_tw=\partial_{xx}w,\ |x|<L,\qquad \frac{d}{dt}w_\pm   =-\frac{1}{\gamma}    \partial_n w|_{\pm L}  +   \Psi_w^\pm \   ,
\label{e:orderscal}
\end{equation}
with 
\[
\Psi_w^\pm=\tau\left[\chi^\varepsilon_{<-\theta}(\partial_n w)(-w)+\chi^\varepsilon_{>\theta}(\partial_n w)(1-w)\right],
\]
and smoothed versions of shifted indicator functions
\begin{align}
     \chi_{>\theta}^\varepsilon(\xi)&=\frac{1}{2}\left(\tanh((\xi-\varepsilon\theta)/\varepsilon)+1)\right), \notag\\
     \chi_{<-\theta}^\varepsilon(\xi)&=\frac{1}{2}\left(\tanh(-(\xi+\varepsilon\theta)/\varepsilon)+1)\right).\label{e:char}
\end{align}
Roughly speaking, the source terms $\Psi_w^\pm$ attempt to quickly ramp up concentrations to 1 or deplete concentrations to 0 depending on the sign of the normal derivative at the boundary, thus robustly restoring an initial weak gradient field to a strong concentration ramp from 0 to 1.  Robin boundary conditions cannot accomplish this robust restoration since an initial adaptation would adjust gradients in the diffusive field to given Dirichlet boundary data, in an initial phase, thus destroying the key gradient information near the boundary.  A linearized analysis in~\cite[\S4.4]{scheel2021signaling} demonstrates that at the boundary of robust recovery, for instance when the residual gradient is too weak, instabilities arise that are initially oscillatory; see also simulations in the supplementary materials, there. It is this oscillatory behavior that we are capturing here on a nonlinear level.

\paragraph{Buffering and delay equations.} As our theorem demonstrates, the dynamic boundary condition adds significant complexity to the dynamics when compared to, say, a nonlinear Robin boundary condition. A simple rationale for this complexity is based on the ability of the diffusive field to store information by effectively acting as a buffer. For instance, an increase in $u^-$ due to the reaction terms at the boundary leads to an increase in the diffusive field $u(t,x)$ that diffusively, hence slowly, spreads in $x$, so that the average of $u$ only slowly adapts to the changed value of $u^-$. This is reflected of course also in the integral terms of a Dirichlet-to-Neumann operator, linking $\partial_n u|_{x=0}$ to given time-dependent Dirichlet data $u|_{x=0}=u_0(t)$, showing history dependence decreasing with $t^{-{1/2}}$. In some ways much simpler history dependence as in $\frac{d}{dt}u(t)=f(u(t),u(t-1))$, or, specifically, the delayed negative feedback  $u'=-\alpha u(t-1)$  is known to lead to Hopf bifurcation and oscillations;
see for instance~\cite{DGVW}. We  comment on more similarities in our discussion.

\paragraph{Bifurcation in the presence of essential spectrum.}
In addition to capturing the buffering effect induced by the coupling of the boundary to the diffusive bulk in the form of emerging oscillations, our analysis resolves a curious technical difficulty in the case $\sigma=0$ and therefore also when $\sigma \sim 0$. The generator $\mathcal{L}$ of the linearized equation 
\[
\frac{d}{dt}u^-=a u^- + b \,\partial_n u|_{x=0}, \qquad u(t,0)=u^-,\qquad \partial_t u = \partial_{xx}u,
\]
defined through 
\[
\mathcal{L}(u^-,u)=(a u^- + b\, \partial_n u|_{x=0},\partial_{xx}u),
\]
as a closed, densely defined operator on $\R\times L^2$, with boundary condition $u(0)=u^-$, possesses essential spectrum $\mathrm{spec}_\mathrm{ess}\,\mathcal{L}=(-\infty,0]$ due to the unbounded domain $x>0$. This essential spectrum cannot be obviously eliminated by the choice of weighted function spaces. As a consequence, the bifurcation due to eigenvalues at $\pm i \omega_\ast$ cannot be reduced dynamically to a finite-dimensional center-manifold in any immediate fashion. Roughly speaking, one expects the oscillations to couple to the neutral diffusive field which induces infinitely many weakly decaying degrees of freedom. It is then rather surprising that an astutely formulated Lyapunov-Schmidt reduction, which we carry out in our proof, is able to mostly ignore the effect of continuous spectrum. The situation is in fact more complicated when considering stability in Section~\ref{sec:stab}. In the stability analysis of the periodic solution, Floquet multipliers are the relevant object, at the bifurcation point given by the eigenvalues of the period map which are both equal to one, due to the Hopf eigenvalues. In addition, diffusion generates essential spectrum in the period map of the form $\exp((-\infty,0])=(0,1]$, so that neutral Floquet multipliers of the periodic orbit generated in the Hopf bifurcation are inherently sitting at the edge of the essential spectrum, with additional stable or unstable Floquet multipliers hidden inside or emerging from the essential spectrum. Again, it is somewhat surprising that our analysis identifies this coupling to be irrelevant, present only through terms that cancel at leading (and thus relevant) order.

We note that a related scenario was also investigated in~\cite{brand2004hopf,kunze2004exchange}, where a localized inhomogeneous term, rather than a boundary condition here,
induces point spectrum in the linear operator. The analysis there is more difficult since nonlinearity is present in the entire domain, so that the authors need to control the interaction between diffusion and possible nonlinear growth in the entire domain. On the other hand, the analysis there is greatly simplified by the presence of an advection term $\partial_x u$, which allows to push the essential spectrum into the negative half plane for the linearized problem. Indeed, adding a drift term $c\,\partial_x u$ in our system~\eqref{eq:sys} allows for a substitution $u=e^{-cx/2}\tilde{u}$ that removes the drift term and induces degradation $\sigma^2=c^2/4$, thus pushing the essential spectrum into the negative half plane.

\paragraph{Outline.}
We prove Theorem~\ref{theorem:hopf} in Section~\ref{s:2} and compute expansions of the bifurcating branch of periodic solutions in Section~\ref{sec:examples_nodegrad}. In particular, we determine the direction of branching and thus the super- or subcritical nature of the bifurcation depending on whether the bifurcating branch coexists with stable or unstable trivial state, respectively. In Section~\ref{sec:stab}, we establish spectral stability and instability of bifurcating periodic solutions for super- and subcritical branches.  Section~\ref{sssec:numericalasympt} contains numerical results based on numerical continuation of an associated periodic boundary-value problem that includes the Dirichlet-to-Neumann operator as a nonlocal term. We in particular corroborate our asymptotic expansions and explore homoclinic limits of the branch of periodic solutions far away from the bifurcation points, reminiscent of observations in delay equations. We conclude with a longer discussion that points to several other settings, in particular in biological contexts,  where mechanisms and phenomena that we distill here in their simplest form may be relevant in the collective dynamics.

\section{Proof of Theorem~\ref{theorem:hopf}}\label{s:2}
The proof proceeds in four  steps. We first reduce to a boundary-integral equation. We then investigate the linearization of this boundary-integral equation, perform the Lyapunov-Schmidt reduction, and finally analyze the reduced equation. 
\paragraph{Reduction to boundary-integral equation.}
We rescale time introducing $\tilde{u}(s, x) := u(\frac{s}{\omega}, x)$ with rescaled argument $s=\omega t$. Equation~\eqref{eq:diff} then becomes
\begin{equation} \label{eq:diffs}
    \omega \partial_s \tilde{u}(s, x) = \partial_{xx} \tilde{u} - \sigma^2\tilde{u}, 
\end{equation}
or, writing $\tilde{u}$ as a Fourier series $\tilde{u}(s, x) = \sum_{\ell\in\bbZ} \hat{u}_\ell(x)  e^{i\ell s}$,
\begin{equation}
    \tilde{u}(s, x) = \sum_{\ell=-\infty}^\infty \hat{u}_\ell(x)  e^{i\ell s},
\end{equation}
we find $(i \omega \ell + \sigma^2) \hat{u}_\ell  = \hat{u}_\ell'' $. We require boundedness of $\hat{u}_\ell(x)$  for $\ell \neq 0$, that is 
\begin{equation}\label{eq:frep}
\hat{u}_\ell(x) = \hat{u}^0_\ell e^{-\sqrt{i\omega \ell + \sigma^2} \;x} \quad \forall \ell \in \mathbb{Z}\setminus\{0\},
 \end{equation}
with standard branch cut.
Substituting this form into the dynamic boundary equation~\eqref{eq:dynb}, we  obtain  time derivative and gradient of $u$ as 
\begin{subequations}
    \begin{eqnarray}\label{eq:def_D}
    \dot{u}_-(t) = \omega \partial_s \tilde{u}(s, 0) =  \sum_{\ell =-\infty}^\infty i\omega \ell  \hat{u}_\ell ^0  e^{i\ell s} &=:& D(\omega) u^-, \\ \partial_n u(t, 0) = \sum_{\ell \in \bbZ\setminus\{0\}} \sqrt{i\omega \ell + \sigma^2} \; \hat{u}_\ell ^0 e^{i\ell s} + \sigma \hat{u}_0^0 &=:& D(\omega, \sigma)^{1/2} u^-.
    \end{eqnarray}
\end{subequations}
Note that $D(\omega, \sigma)^{1/2}$ is a nonlocal pseudo-differential operator, usually referred to as the Dirichlet-to-Neumann operator as it relates Dirichlet data $u(t,0)$ to Neumann data $\partial_n u(t,0)$, in our case for the heat operator in a semi-infinite strip. Its Fourier multiplier symbol is smooth in $(\sigma, \omega)$, also near $\sigma = 0$.

Having thus solved the equation in the bulk, we can restrict to the time-periodic boundary-integral equation
\begin{equation}\label{eq:bi}
    \tilde{F}(u^-, \mu, \omega, \sigma) :=D(\omega) u^- - f(u^-, D(\omega, \sigma)^{1/2} u^-, \mu) = 0,
\end{equation}
in the sense that any solution of~\eqref{eq:bi} yields $u(s,0)$ by~\eqref{eq:bc}, then $\tilde{u}(s,x)$ through~\eqref{eq:frep}, and $u(t,x)$ from rescaling $s=\omega t$. 

It is convenient to simplify the equation in a first step by shifting the equilibrium from~\eqref{e:eqb} to the origin, setting $u^-=u_\ast^-(\mu, \sigma)+v^-$, with the $s$-independent function $u_\ast^-$, to find the equivalent equation
\begin{equation}\label{eq:bi2}
    F(v^-, \mu, \omega, \sigma) :=D(\omega) v^- - f(u_\ast^-(\mu, \sigma) + v^-, \sigma u_\ast^-(\mu, \sigma) +D(\omega, \sigma)^{1/2} v^-, \mu) = 0,
\end{equation}
We collect some basic properties of $F$. 

\begin{lemma}
    The function $F(v^-, \mu, \omega, \sigma)$ is a smooth map from $H^2(S^1, \bbR) \times \bbR \times \bbR \times \bbR \to H^1(S^1, \bbR)$, where $S^1=\bbR/(2\pi\bbZ)$. In addition, $F(0, \mu, \omega, \sigma) = 0$  since $f(u_\ast^-(\mu, \sigma), \sigma u_\ast^-(\mu, \sigma), \mu)=0$. The derivative with respect to the first argument at $(0, \mu, \omega, \sigma)$ is the linear operator
    \begin{eqnarray*}
        \partial_1 F(0, \mu, \omega, \sigma) &=& D(\omega) - \partial_1 f_\ast - \partial_2 f_\ast  D(\omega, \sigma)^{1/2}  
    \end{eqnarray*}
    where $f$ and its derivatives, denoted here by $\partial_1 f_\ast$ and $\partial_2 f_\ast$,    
    are evaluated at $u^- = u_\ast^-(\mu, \sigma)$.
\end{lemma}
\begin{proof}
    The operators $D(\omega)$ and $D(\omega, \sigma)^{1/2}$ are smooth, in fact analytic in $\omega$ as bounded maps from $H^2$ into $H^1$ or $H^{3/2}$, respectively, as one can readily see by finding the (complex) derivative in $\omega$ and $\sigma$. On the other hand, the superposition operator induced by $f$ is smooth as a map on $H^{3/2}$ which is an algebra.
\end{proof}

\paragraph{Properties of the linearization.}

The derivative with respect to the first component, evaluated at the critical point $(0, 0, \omega_\ast, \sigma_\ast)$, turns out not to be invertible and therefore indicates a bifurcation point, due to Assumption~\ref{assu:degen}. The next lemma summarizes properties of this derivative viewed as a linear operator $\cL :=\partial_1 F(0, 0, \omega_\ast, \sigma_\ast):H^2(S^1,\bbR)\to H^1(S^1,\bbR)$.
\begin{lemma}[Fredholm  properties of $\cL$] \label{lem:kerL}
    The linear operator $\cL$ has Fredholm index 0 and its kernel is spanned by $e^{is}$ and $e^{-is}$ so that $\dim_\bbR(\mathrm{ker}(\cL)) = 2$. Furthermore, $\mathrm{range}(\cL)^\perp = \ker(\partial_1 F(0, 0, \omega_\ast, \sigma))$, where $\perp$ refers to the $L^2$-scalar product. 
\end{lemma}
\begin{proof} 
    The principal part $D(\omega_\ast)$ of $\cL$ minus the identity, $D(\omega_\ast)-1$, is bounded invertible, as can be seen directly using Fourier series, hence Fredholm of index 0. 
    The remaining terms of $\mathcal{L}$ are compact since they can be considered as bounded operators from $H^2$ into $H^{3/2}$, which in turn is compactly embedded into $H^1$, 
    so that $\cL$ is Fredholm of index 0. Any element in the kernel of $\cL$ can be written as a Fourier series. The action of $\cL$ on Fourier coefficients is diagonal, 
    \begin{equation*}
        \cL e^{i\ell s} = \left(i\omega_\ast \ell - \partial_1 f_\ast - \partial_2 f_\ast \sqrt{i\omega_\ast \ell + \sigma_\ast^2}\right) e^{i\ell s} = \dm(i\omega_\ast\ell; 0, \sigma_\ast) e^{i\ell s},\ \text{ with }\dm \text{ from \eqref{eq:dlambdaIFT},}
    \end{equation*}
    so that  $\cL e^{i\ell s}=0$ by Assumption~\ref{assu:degen} when $|\ell|=1$. The kernel of $\cL$ therefore is spanned by $e^{\pm i s}$, which, restricting to real functions, gives the result. The $L^2$-adjoint of $\cL$ possesses the same form, replacing simply $\ell$ by $-\ell$, so that the kernel of the adjoint and hence the orthocomplement of the range are identical to the kernel of $\cL$.
\end{proof}
\begin{remark}[Non-resonance]\label{r:nonres}
Inspecting the proof, one notices that it is sufficient to assume that $\dm(i\ell\omega_\ast)\neq 0$ for $\ell\in\Z$, $|\ell|\neq 1$, rather than for $\ell\in\R$, $|\ell|\neq 1$. The remainder of the proof then indeed goes through. The stronger condition (2) in Assumption~\ref{assu:degen} is needed in the stability proof, Section~\ref{sec:stab}.
\end{remark}

\paragraph{Lyapunov-Schmidt splitting and reduction.}

We denote the two-dimensional kernel of $\cL$ mentioned in Lemma~\ref{lem:kerL} by $E_c^2= \mathrm{span}\{\cos(s), \sin(s)\}$, and we write  $E_h^2$ for its $L^2$-orthocomplement, so that we have $H^2(S^1, \bbR) = E_c^2 \oplus E_h^2$. We can consider the two-dimensional kernel of $\cL$ and its complement as subspaces of $H^1(S^1, \bbR)$ as well, denoting them respectively by $E_c^1$ and $E_h^1$ with $H^1(S^1, \bbR)= E_c^1 \oplus E_h^1$.
We also introduce the associated projections,
\begin{subequations}
    \begin{eqnarray}
        P_1: H^1(S^1, \bbR) \to H^1(S^1, \bbR), \quad P_1 E_c^1 = E_c^1, \quad P_1 E_h^1 = \{0\}, \\
        P_2: H^2(S^1, \bbR) \to H^2(S^1, \bbR), \quad P_2 E_c^2 = E_c^2, \quad P_2 E_h^2 = \{0\}.
    \end{eqnarray}
\end{subequations}
By Lemma~\ref{lem:kerL}, we have that $\range(P_2) = \ker(\cL)$ and $\ker(P_1) = \range(\cL)$. Correspondingly, we also split the variable $v^-$ into
\begin{equation}
    v^- = v^-_c + v^-_h, \quad v^-_c\in E_c^2, \; v^-_h \in E_h^2
\end{equation}
and decompose the nonlinearity as
\begin{equation}
\begin{aligned}
    G_c(v^-_c, v^-_h; \mu, \omega, \sigma) &:= P_1 F(v^-_c + v^-_h, \mu, \omega, \sigma):&\ E_c^2\times E_h^2\times\bbR^2\to E_c^1,    \\
    G_h(v^-_c, v^-_h; \mu, \omega, \sigma) &:= (1-P_1) F(v^-_c + v^-_h, \mu, \omega, \sigma):&\ E_c^2\times E_h^2\times\bbR^2\to E_h^1.
\end{aligned}
\end{equation}
as well as 
\begin{equation}
    \begin{aligned}
        \partial_1 G_h(0, 0; \mu, \omega, \sigma) &= (1-P_1) \cL(\mu, \omega, \sigma):& \; E_c^2 \to E_h^1,\\ \quad \partial_2 G_h(0, 0; \mu, \omega, \sigma) &= (1-P_1) \cL(\mu, \omega, \sigma): &\; E_h^2 \to E_h^1.
    \end{aligned}
\end{equation}
where $\cL(\mu, \omega, \sigma) := \partial_1 F(0, \mu, \omega, \sigma)$.
Hence, at $\mu = 0$, $\sigma = \sigma_\ast$, and $\omega=\omega_\ast$,
\begin{equation}
    \partial_1 G_c = 0, \quad \partial_2 G_c = 0, \quad \partial_1 G_h = 0 .
\end{equation}
In the following, we fix $\sigma = \sigma_\ast$ and suppress $\sigma$-dependence in the notation since it simply adds a free parameter to the results. Since $\cL$ is Fredholm of index 0, a bordering theory for Fredholm operators implies that $\partial_2 G_h$ is Fredholm of index 0 (cf.~\cite{sandstede2008relative},~\cite{teschl2020topics}). Since moreover $\partial_2 G_h$ has trivial kernel in $E_h^2$, it is in fact bounded invertible. We may then write our equation~\eqref{eq:bi2} as the system
\begin{eqnarray}
    \begin{pmatrix}
        G_c \\ G_h
    \end{pmatrix}
    = 
    \begin{pmatrix}
        0 &0 \\
        0 &\partial_2 G_h
    \end{pmatrix}
    \begin{pmatrix}
        v^-_c \\
        v^-_h
    \end{pmatrix}
    + A_\omega (\omega - \omega_\ast) 
    \begin{pmatrix}
        v^-_c \\ v^-_h
    \end{pmatrix}
    + A_\mu \mu
    \begin{pmatrix}
        v^-_c \\ v^-_h
    \end{pmatrix}
    + \cO(|v^-_c|^2 + |v^-_h|^2).
\end{eqnarray}
with  linear operators $A_\omega$ and $A_\mu$. The invertibility of $\partial_2 G_h$ at $\mu=0$ now allows us to use the Implicit Function Theorem for the equation $G_h(v^-_c, v^-_h; \mu, \omega) = 0$ and solve for $v^-_h$ as a unique function of the remaining parameters $(v^-_c, \mu, \omega)$, that is,  to find the existence of a smooth and unique function $\psi(v^-_c, \mu, \omega)$, defined in a neighborhood of the origin, such that $G_h(v^-_c,\psi(v^-_c, \mu, \omega); \mu,\omega)=0$. 
From uniqueness and the fact that $F(0,\mu,\omega)=0$, we conclude that $0 = \psi(0, \mu, \omega)$, so that $\psi \in \cO(|\mu v^-_c|, |(\omega-\omega_\ast)v^-_c|, |v^-_c|^2)$. 

\paragraph{Reduced equation and proof of Theorem~\ref{theorem:hopf}. }
It remains to solve $G_c=0$, with $v^-_h$ given through $\psi$. We now choose coordinates in the kernel setting  $v^-_c = r\cos(s)$ and thus fixing a time-$s$ translate. We find, using the basis $\{e^{i\ell s}, e^{-i\ell s}\}$ of its range together with perturbation $v^-_c = \frac{r}{2}e^{is} + \frac{r}{2} e^{-is}$,
\begin{equation}
    g_c(r, \mu, \omega, \sigma) := G_c(r\cos(\cdot), \psi(r\cos(\cdot); \mu, \omega, \sigma); \mu, \omega, \sigma) = g^c_1 e^{is} + g_2^c e^{-is} \in  E_c^1\sim \bbR^2 \; ,
\end{equation}
and, of course, $g_1^c = \overline{g_2^c}$. 
Since $F(0,\mu,\omega, \sigma)=0$, we immediately find that $g_c(0, \mu, \omega, \sigma) = 0$. Expanding about the bifurcation point, we find 
\begin{equation}\label{e:gcred}
        g_1^c(r,\mu,\omega, \sigma) = \partial_2 \dm_\ast \cdot \frac{r}{2}\mu + i\; \partial_1 \dm_\ast \cdot \frac{r}{2} (\omega-\omega_\ast) + \cO\left(r^3+\mu^2r+(\omega-\omega_\ast)^2r\right),
\end{equation}
where ``$\ast$'' denotes evaluation at $(\lambda, \mu) = (i\omega_\ast, 0)$. Dividing by the trivial factor $r$, invoking Hadamard's lemma, we can solve for $(\mu,\omega)$ near $(0, \omega_\ast)$ and $r=0$ as a function of the parameter $r$ using the Inverse Function Theorem if $\partial_2 \dm$ and $i\; \partial_1 \dm$, evaluated at $(\mu, \sigma, \omega) = (0, \sigma_\ast, \omega_\ast)$ are linearly independent over $\bbR$ as elements of $\bbC$. To establish that fact, suppose the contrary is true, that is, $\partial_2\dm_\ast = \rho\cdot i\;\partial_1\dm_\ast$ for some $\rho\in\bbR$. Then, using that 
$\partial_\mu \lambda_\ast = -\partial_2\dm_\ast / \partial_1 \dm_\ast$ by implicit differentiation,  
\begin{equation*}
    \Re\left(\partial_\mu\lambda_\ast\right) = -\Re\left(\frac{\rho \,i\; \partial_1\dm_\ast \overline{\partial_1\dm_\ast}}{\partial_1\dm_\ast\overline{\partial_1\dm_\ast}} \right) = 0,
\end{equation*}
hence a contradiction to Assumption~\ref{assu:bif}.

In summary, we can solve $\frac{1}{r}g_1^c = 0$ with the Inverse Function Theorem for $\mu = \mu(r, \sigma)$ and $\omega = \omega(r, \sigma)$. 

 Writing $\mu=\mu_{\mathrm{nl}}(r, \sigma)$ and $\omega=\omega_{\mathrm{nl}}(r, \sigma)$ for the solutions, our perturbation solution is then of the form 
\[
v^-(s;r, \sigma)=r\cos(s)+\psi(r\cos(\cdot),\mu_\mathrm{nl}(r, \sigma),\omega_\mathrm{nl}(r, \sigma), \sigma)(s).
\]
Substituting back, this yields $u^-(t;r, \sigma)=u^-(\mu_\mathrm{nl}(r, \sigma))+v^-(\omega_\mathrm{nl}(r, \sigma)t;r, \sigma)$, at the parameter value $\mu=\mu_\mathrm{nl}(r,\sigma)$.

In order to establish that $\omega_\mathrm{nl}'(0, \sigma)=0$ and $\mu_\mathrm{nl}'(0, \sigma)=0$,
we note that 
\begin{eqnarray*}
v^-(s;-r, \sigma)&=&-r\cos(s)+\psi(-r\cos(\cdot),\mu_\mathrm{nl}(r, \sigma),\omega_\mathrm{nl}(r, \sigma), \sigma)(s)\\
&=&r\cos(s+\pi)+\psi(r\cos(\cdot+\pi),\mu_\mathrm{nl}(r, \sigma),\omega_\mathrm{nl}(r, \sigma), \sigma)(s)\\
&=&r\cos(s+\pi)+\psi(r\cos(\cdot),\mu_\mathrm{nl}(r, \sigma),\omega_\mathrm{nl}(r, \sigma), \sigma)(s+\pi)\\
 &=& v^-(s+\pi;r, \sigma),
\end{eqnarray*}
where in the third equality we used uniqueness of the solution. In particular, there is a solution with the same frequency and parameter value for $-r$, which by uniqueness implies that $\mu_\mathrm{nl}$ and $\omega_\mathrm{nl}$ are even, with vanishing derivatives at the origin as claimed. 

Reconstructing $u(t,x)$ from $u^-$, we see that $u_\infty=\lim_{x\to\infty}u(t,x)\equiv\hat{u}^0_0$, that is, the limit is given by the zeroth Fourier mode, which belongs to $E_h^2$. 
Exponential asymptotics as in~\eqref{e:asy}  follows from the Fourier representation~\eqref{eq:frep}.

\section{Leading-order expansions in an example} \label{sec:examples_nodegrad}

We specialize our setting to boundary kinetics $f$ that are linear in $\partial_n u$ and allow for a trivial equilibrium $u_*\equiv 0$, $u^-_*\equiv 0$, for all parameter values,
\begin{equation} \label{eq:sys_cubicorder_degrad}
    \begin{array}{rcll}
        \partial_t u &=&  \partial_{xx}u - \sigma^2 u, \\
        \dot{u}^-(t) &=& -\alpha u^- + \beta (u^-)^2 + \gamma (u^-)^3 + (\mu + \sqrt{2\alpha}) \; \partial_n u(t, 0),  &  \\
        u^-(t) &=& u(t, 0)
    \end{array}
\end{equation}
with parameters $\alpha>0$, $\beta,\gamma\in\R$, and bifurcation parameter $\mu\in\R$, so that $f(u^-(t), \partial_n u(t,0), \mu) = -\alpha u^- + \beta (u^-)^2 + \gamma (u^-)^3 + (\mu + \sqrt{2\alpha})\; \partial_n u(t,0)$ in~\eqref{eq:sys}.

The choice of coefficient $\mu+\sqrt{2\alpha}$ for $\partial_n u$ ensures that there is in fact a Hopf bifurcation at $\mu=0$. 

\begin{lemma}[Linear Hopf bifurcation]\label{l:linass}
Assume $\alpha>2\sigma_\ast^2\geq 0$. 
Then the system~\eqref{eq:sys_cubicorder_degrad} satisfies Assumptions~\ref{assu:equi}---\ref{assu:bif} with $u_\ast^-(\mu,\sigma)\equiv 0$, $\omega_\ast(\sigma_\ast)=\sqrt{\alpha(\alpha -2\sigma_\ast^2)}$.   In particular, $\Re\partial_\mu\lambda_\ast>0$.
\end{lemma}
\begin{proof}
    We clearly have the equilibrium $u_\ast^-=0$ and therefore need to check the linear assumptions. We have in this case 
    \begin{equation}\label{e:dm}
    \dm(\lambda;\mu, \sigma)=\lambda+\alpha - (\mu + \sqrt{2\alpha}) \; \sqrt{\lambda + \sigma^2}=0. 
    \end{equation}
    Solving $\dm(i \omega;0,\sigma)=0$, we quickly find 
    \[
    \alpha^2 +i2\alpha\omega - \omega^2 = 2\alpha(\sigma_\ast^2 + i\omega).
    \]
    Taking real and imaginary parts gives $\omega=\omega_\ast=\sqrt{\alpha(\alpha-2\sigma_\ast^2)}$ and establishes  (1) and (2) in~\eqref{e:asspt12}. We also find that
    \[
    \partial_1 \dm(i\omega_\ast;0,\sigma_\ast)=1-\frac{1}{2}\sqrt{2\alpha}(i\omega_\ast +\sigma_\ast^2)^{-1/2},
    \]
    which is real only when $\omega_\ast=0$, thus implying (3) in~\eqref{e:asspt12}. It remains to establish strict crossing.    We therefore set  $\lambda + \sigma^2=\gamma^2$, so that roots of $\dm$ solve  $\gamma^2 - \sigma^2+\alpha - (\mu + \sqrt{2\alpha}) \; \gamma=0$. Differentiating gives, at $\mu=0$, 
    \[
    \partial_\mu\lambda=2\gamma\partial_\mu\gamma=\frac{\gamma^2}{\gamma-\sqrt{\alpha/2}}.
    \]
    Since the denominator is nonzero, this expression is purely imaginary if $\gamma^2=i \tau (\gamma-\sqrt{\alpha/2})$ for some $\tau\in\R$, which after squaring gives
    \[
    \gamma^4=-\tau^2\left(\gamma^2+\frac{\alpha}{2} -\sqrt{2\alpha}\gamma\right)=\tau^2\left(\frac{\alpha}{2}-\sigma^2\right),
    \]
    where we used the equation for $\gamma$ in the last equality. The right-hand side is positive, which implies that $\gamma^2\in\R$, necessarily, which, using $\lambda+\sigma^2=\gamma^2$ and $\lambda\in i \R$, implies $\lambda=0$, a contradiction.
    
    In order to establish the sign of $\Re\partial_\mu \lambda$, it is then sufficient to evaluate at $\sigma=0$,  when $\partial_\mu\lambda=\frac{i\alpha}{\sqrt{i\alpha}-\sqrt{\alpha/2}}$, so that $\Re\partial_\mu\lambda>0$.
\end{proof}
We now turn to the nonlinear equation~\eqref{eq:sys_cubicorder_degrad} with $\beta$ and $\gamma$ possibly nonzero. By Lemma~\ref{l:linass}, the assumptions of Theorem~\ref{theorem:hopf} hold and we have a branch of solutions with
\begin{equation}\label{e:mu2def}
\mu=\mu_2 r^2+\cO(r^4),\qquad \qquad  
\omega=\omega_\ast(\sigma_\ast)+\omega_2 r^2+\cO(r^4),\qquad \text{and}
\qquad u_\infty=u_{\infty,2} r^2+\cO(r^4).
\end{equation}
When computing the expansions, the action of the linearization on Fourier modes contributes key coefficients. We therefore define coefficients as 
\begin{equation}\label{e:lincoeff}
    \Lambda_\ell=\dm(i\omega_\ast \ell;0,\sigma_\ast), \qquad \Lambda_{1,\mu}=\-\partial_2\dm(i\omega_\ast \ell;0,\sigma_\ast), \quad \text{and}\quad  \Lambda_{1,\omega}=i\partial_1 \dm(i\omega_\ast ;0,\sigma_\ast).
\end{equation}

\begin{theorem}[Expansion of bifurcating branch] \label{cor:m_w_cubicorder_degrad}
    Fix $\sigma_\ast\geq 0$ and consider~\eqref{eq:sys_cubicorder_degrad} with $\mu\sim 0$. Recall the definition of $\dm$ in~\eqref{e:dm}. Then the bifurcating branch has the expansion~\eqref{e:mu2def} with
    \begin{equation}   \label{e:exp}
        \mu_2= \frac{\Im(M \overline{\Lambda_{1,\omega}})}{\Im(\Lambda_{1,\omega} \overline{\Lambda_{1,\mu}})},   \qquad
        \omega_2=\frac{\Im(M \overline{\Lambda_{1,\mu}})}{\Im(\Lambda_{1,\omega} \overline{\Lambda_{1,\mu}})},\qquad
        u_{\infty,2}= \frac{\beta}{2\Lambda_0},
    \end{equation}
    where we used the definitions in~\eqref{e:lincoeff} and the shorthand
    \[
    M=  \frac{3}{4}\gamma + \beta^2(\Lambda_0^{-1}+(2\Lambda_2)^{-1}).
    \]
\end{theorem}
\begin{proof}
    We expand the perturbation  $v$ in the parameter $r$ and in a Fourier series as
    \begin{equation}
        v(s;r)=\frac{r}{2} e^{i s}+\frac{r}{2} e^{-i s} + v_{2,2}r^2 e^{2 is}+v_{2,0}r^2 +v_{2,-2}r^2 e^{-2 is}+ \sum_{j=-3}^3 v_{3,j}r^3 e^{i j s}.
    \end{equation}
    Clearly, $v_{\ell,j}=\overline{v_{\ell,-j}}$ since $v$ is real. 
    Using the terminology from~\eqref{e:lincoeff}, we find at order $r^2$ the equations in the Fourier modes $e^{2 i s},\ 1,$ and $e^{-2is}$, respectively, 
    \begin{equation}\label{e:v2}
        \begin{aligned}
            \Lambda_2 v_{2,2} - \frac{1}{4}\beta&=0,\\
            \Lambda_0 v_{2,0} - \frac{1}{2}\beta&=0,\\
            \Lambda_{-2} v_{2,-2} - \frac{1}{4}\beta&=0,            
        \end{aligned}
    \end{equation}
    resulting in 
    \begin{equation}
        v_{2,2}=\overline{v_{2,-2}}=\beta/(4\Lambda_2),\qquad v_{2,0}=\beta/(2\Lambda_0).
    \end{equation}
    At order $r^3$ we find the equation for the mode $e^{i s}$ in the form 
    \begin{equation}\label{e:O3}
        \frac{1}{2}\Lambda_{1,\mu}\mu_2 +  
        \frac{1}{2}\Lambda_{1,\omega}\omega_2 -  
        \beta(v_{2,0}+v_{2,2})-\frac{3}{8}\gamma =0.
    \end{equation}
    Setting 
    \[
    M= \beta(2v_{2,0}+2v_{2,2})+\frac{3}{4}\gamma =  \beta^2(\Lambda_0^{-1}+(2\Lambda_2)^{-1})+\frac{3}{4}\gamma,
    \]
    we can then solve for $\mu_2$ and $\omega_2$. We multiply~\eqref{e:O3} by $\overline{\Lambda_{1,\omega}}$ and  $\overline{\Lambda_{1,\mu}}$, respectively, then taking real parts, to find the expansions in~\eqref{e:exp}
    \begin{equation}
            \omega_2=\frac{\Im(M \overline{\Lambda_{1,\mu}})}{\Im(\Lambda_{1,\omega} \overline{\Lambda_{1,\mu}})},\qquad \mu_2=\frac{\Im(M \overline{\Lambda_{1,\omega}})}{\Im(\Lambda_{1,\omega} \overline{\Lambda_{1,\mu}})}.         
    \end{equation}
    The coefficient of $u_\infty$ is simply $v_{2,0}$, corresponding to the zeroth Fourier mode which is the only mode with possibly neutral decay. 
\end{proof}
\begin{cor}[Expansions without degradation]\label{cor:m_w_cubicorder}
    Setting $\sigma=0$ in the expansions in Theorem~\ref{cor:m_w_cubicorder_degrad}, we find explicitly for the coefficients from~\eqref{e:mu2def},
      \begin{equation}   \label{e:exp2}
        \mu_2= \frac{1}{\sqrt{\alpha}} \left( 
        \frac{\beta^2}{\alpha}\left(\frac{1}{3}-\frac{1}{2\sqrt{2}}\right) -\frac{3}{4\sqrt{2}}\gamma\right),  \qquad
        \omega_2= -\left( \frac{7\beta^2}{6\alpha} + \frac{3\gamma}{4}\right)  ,\qquad
        u_{\infty,2}=\frac{\beta}{2\alpha}.   
    \end{equation}
\end{cor}
\begin{proof}
    The proof is a direct evaluation of the terms in Theorem~\ref{cor:m_w_cubicorder_degrad}.
\end{proof}
\begin{remark}[Expansion in trigonometric functions to cubic order]
    At cubic order, a more detailed calculation gives the expansion
    \begin{equation}
        \begin{array}{ll}
             v^- = &r\cos(s) + r^2\frac{\beta}{2\alpha}\left(1 - \frac{1}{3}(1+\sqrt{2}) \cos(2 s) + \frac{1}{3}(2+\sqrt{2}) \sin(2 s)\right)  \\
             &+r^3\left((-\frac{1}{48}(5+2\sqrt{2}+\sqrt{3})\frac{\beta^2}{\alpha} - \frac{1}{32}(1+\sqrt{3})\gamma) \cos(3s) \right. \\
             & \qquad \left. + (-\frac{1}{48}(5+4\sqrt{2} + 3\sqrt{3} + 2\sqrt{6})\frac{\beta^2}{\alpha} + \frac{1}{32}(3+\sqrt{3})\gamma) \sin(3s)\right) + \cO(r^4).
        \end{array}
    \end{equation}
\end{remark}
We note here that $\mu_2>0$ corresponds to a supercritical Hopf bifurcation, $\mu_2<0$ to a subcritical Hopf bifurcation, since the trivial branch is unstable for $\mu>0$ as we found $\Re\partial_\mu\lambda>0$ in Lemma~\ref{l:linass}. Corollary~\ref{cor:m_w_cubicorder} then gives an explicit expression 
\begin{equation}\label{e:gamcrit}
\gamma_\mathrm{crit}=\frac{\beta^2}{\alpha}\left(\frac{4\sqrt{2}}{9} - \frac{2}{3}\right)\sim -0.0381\frac{\beta^2}{\alpha} ,
\end{equation}
so that the Hopf bifurcation is supercritical precisely when $\gamma<\gamma_\mathrm{crit}<0$. In particular, $\gamma=0$, that is, quadratic terms only, render the bifurcation subcritical.

\section{Stability and instability in an example}\label{sec:stab}

We analyze stability of the bifurcating periodic orbits found in Corollary~\ref{cor:m_w_cubicorder}, without degradation, and discuss positive degradation in Remark~\ref{r:posd}, below. We therefore  consider perturbations of our branch of periodic solutions $u_\ast^-(s;r)=r\cos(s)+\cO(r^2)$, that is solutions with infinitesimally small perturbations of the form 
\[
u^-(s; r)=u_\ast^-(s;r)+\varepsilon e^{\lambda s}v^-(s),\qquad v^-(s)=v^-(s+2\pi),
\]
which gives at order $\varepsilon$, the linear boundary-value problem
\begin{equation}
    \label{e:Llinbvp}
    \mathcal{L}(\lambda;r)v^-:=\omega(\frac{d}{ds}v^-+\lambda v^-)+\alpha v^- - 2\beta u_\ast^- v^- - 3\gamma (u_\ast^-)^2 v^- - (\mu+\sqrt{2\alpha})D^{1/2}(\omega,\lambda)v^-=0.
\end{equation}
Here, we use $\omega=\omega(r)$ and $\mu=\mu(r)$  from~\eqref{e:exp}, so that $\mathcal{L}$ indeed depends on $r$ and $\lambda$, only. As usual in Floquet stability questions, we can restrict to $|\Im\lambda|\leq \pi$, the Brillouin zone. The Dirichlet-to-Neumann operator $D^{1/2}(\omega,\lambda)$ is now given through the Fourier symbol
\begin{equation}\label{e:d1/2}
D^{1/2}(\omega,\lambda)e^{i \ell s}=e^{i\ell s}\cdot\left\{\begin{array}{ll}
\sqrt{\omega(i\ell +\lambda)},&\ell\neq 0,\\
\sqrt{\omega}\sqrt{\lambda}, & \ell=0.
\end{array}\right.
\end{equation}
We say that the periodic solution is spectrally stable if all solutions to~\eqref{e:Llinbvp} have $\Re\lambda\leq 0$ and spectrally unstable otherwise. 
\begin{theorem}\label{t:stab}
In the setting of Theorem~\ref{cor:m_w_cubicorder_degrad}, we have that the periodic solutions with $r>0$ sufficiently small are spectrally stable if $\mu_2>0$ and spectrally unstable if $\mu_2<0$. 
\end{theorem}
\begin{proof}
The strategy is to expand $\mathcal{L}$ from~\eqref{e:Llinbvp} in $r$ and $\lambda$. We therefore start by studying properties of $\mathcal{L}$ for $r=0$,
\[
\mathcal{L}(\lambda;0)v^-=\alpha(\frac{d}{ds}v^-+\lambda v^-)+\alpha v^- - \sqrt{2\alpha}D^{1/2}(\alpha,\lambda)v^-,\qquad \mathcal{L}(\lambda;0)e^{i\ell s}=\dm(\alpha(i\ell+\lambda))e^{i\ell s}.
\]
Restricting to $\Re\lambda\geq 0$, we then find that  the kernel of $\mathcal{L}(0;\lambda)$ is nontrivial precisely at $\lambda=0$ (recalling that $|\Im\lambda|\leq \pi $) with two-dimensional kernel and cokernel both spanned by $e^{\pm i s}$. Note that we here need the full strength of Assumption~\eqref{assu:degen} (2), as announced in Remark~\ref{r:nonres}.

Mimicking the Lyapunov-Schmidt reduction for the nonlinear problem, we can therefore determine stability by identifying the values of $\lambda$ in a neighborhood of the origin for which there is a solution to $\mathcal{L}(\lambda;r)v^-=0$. Note the small technical difficulty caused by the fact that $D^{1/2}$ is not smooth in $\lambda$, which we remedy setting $\lambda=\rho^2$ in 
\eqref{e:d1/2}. 

We then set 
$v^-=A e^{is}+\bar{A} e^{-i s} +V^-_h$, with $V^-_h$ representing a complement of the kernel, that is,
Fourier modes with $|\ell|\neq 1$. We then solve the equation
$\mathcal{L}(\lambda;r)v^-=0$ in the range of $\mathcal{L}(0;0)$ by inverting the linear operator, thus finding $V^-_h=v_+^h A + v_-^h \bar{A}$. Here, $v_\pm^h=v_\pm^h(s;r,\lambda)$ are perpendicular to the kernel and thus can be expanded in Fourier series
\[
v_\pm^h=\sum_{|\ell|\neq 1}v_\pm^\ell e^{i\ell s}.
\]
Writing $P_{\pm1}$ as the projection onto the Fourier modes $e^{\pm i s}$, respectively, the existence of eigenvalues is then equivalent to solving 
\begin{equation}\label{e:redabs}
    P_{\pm1}\mathcal{L}(\lambda;r)(A e^{is}+\bar{A}e^{-is}+v_+^h A + v_-^h \bar{A})=0,
\end{equation}
which in turn, introducing coordinates in the range of $P_1+P_{-1}$ through linear combinations of $e^{\pm i s}$, is a linear equation in $(A,\bar{A})$ that can be written in  complex matrix  form as 
\begin{equation}\label{e:Lred}
0=\mathcal{L}_\mathrm{red}(\lambda;r)\left(\begin{array}{c}
A\\ \bar{A}
\end{array}
\right)
=\left(\begin{array}{cc}
L_{11}(\lambda;r)&L_{12}(\lambda;r)\\
L_{21}(\lambda;r)&L_{22}(\lambda;r)
\end{array}
\right)\left(\begin{array}{c}
A\\ \bar{A}
\end{array}
\right).
\end{equation}
Key to establishing stability and instability then is to compute expansions of the entries $L_{jk}$ in terms of $r$ and $\lambda$. From the discussion above, we have that $L_{kj}(0;0)=0$ for all $k,j$, and
\begin{equation}\label{e:lam-terms}
 L_{12}(\lambda;0)=L_{21}(\lambda;0)=0,\quad L_{11}(\lambda;0)= \frac{1+i}{2}\alpha\lambda+\cO(\lambda^2),\ L_{22}(\lambda;0)= \frac{1-i}{2}\alpha\lambda+\cO(\lambda^2).
\end{equation}
We recall that due to $\mathcal{L}(\lambda;0)1=\dm(\alpha\lambda)=\alpha(\lambda+1 -\sqrt{2\lambda})$, we expect that the expansion of the coefficients $L_{jk}$ is smooth in $\sqrt{\lambda}$ rather than $\lambda$. We shall indeed identify terms of the form $r^2\sqrt{\lambda}$ in all matrix entries, below.

In order to establish expansions in $r$, recall~\eqref{e:v2} for quadratic coefficients $v_{2,j}$, $j \in \{-2,0,2\}$, in the expansion of the nonlinear solution and the expansions for $\omega$ and $\mu$ from~\eqref{e:mu2def}. We then find the somewhat cumbersome expression for the equation on the kernel,
\begin{equation}\label{e:P1L}
\begin{aligned}
P_1\mathcal{L}(0;r)&\left(A e^{is}+\bar{A}e^{-is}+\sum_{|\ell|\neq 1} \left(v_+^\ell A e^{i\ell s}+v_-^\ell \bar{A} e^{i\ell s}\right)\right)=\\
=&r^2\left[\omega_2 i A - \mu_2\sqrt{i\alpha} A - \sqrt{2\alpha}\frac{\sqrt{i}}{2\sqrt{\alpha}}\omega_2 A - 2 \beta (v_{2,2} \bar{A}+v_{2,0}A) - \frac{3\gamma}{4}(2A+\bar{A}) \right]\\
&+r[-\beta (v_+^0 A + v_-^0 \bar{A})-\beta (v_+^2 A + v_-^2 \bar{A})]
\\
&+\cO(r^3).
\end{aligned}
\end{equation}
In order to expand to quadratic order, we therefore need to find $v_\pm^{0/2}$, which are determined by the equations projected via $P_0$ on constants and $P_{\pm2}$ on $e^{\pm2is}$.

We find 
\begin{equation}\label{e:P0L}
\begin{aligned}
P_0\mathcal{L}(0;r)&\left(A e^{is}+\bar{A}e^{-is}+\sum_{|\ell|\neq 1} \left(v_+^\ell A e^{i\ell s}+v_-^\ell \bar{A} e^{i\ell s}\right)\right)=\\
=&\alpha (v_+^0 A + v_-^0\bar{A})-\beta r (A+\bar{A})-\sqrt{2\alpha}\sqrt{\alpha}\sqrt{\lambda}(v_+^0A+v_-^0\bar{A})+\cO(|\lambda|+r^2),
\end{aligned}
\end{equation}
which readily yields 
\begin{equation}\label{e:h0}
v_\pm^0=\frac{\beta}{\alpha}(1+\sqrt{2}\sqrt{\lambda})r+\cO(r^2+|\lambda|).
\end{equation}
Similarly, using the shorthand $\Lambda_2=\alpha(1-{\sqrt{2}}+(2-{\sqrt{2}})i)$ from~\eqref{e:lincoeff},
\begin{equation}\label{e:P2L}
\begin{aligned}
P_2\mathcal{L}(0;r)&\left(A e^{is}+\bar{A}e^{-is}+\sum_{|\ell|\neq 1} \left(v_+^\ell A e^{i\ell s}+v_-^\ell \bar{A} e^{i\ell s}\right)\right)=\\
=&\Lambda_2 (v_+^2 A + v_-^2\bar{A})-\beta r A+\cO(|\lambda|+r^2),
\end{aligned}
\end{equation}
which gives, equating coefficients of $A$ and $\bar{A}$, separately, 
\begin{equation}\label{e:h2}
v_-^2=\cO(r^2+|\lambda|),\qquad v_+^2=r\frac{\beta}{\Lambda_2}+\cO(r^2+|\lambda|).
\end{equation}
Using the expressions~\eqref{e:h0} and~\eqref{e:h2} in~\eqref{e:P1L}, we find expansions for the entries $L_{jk}$ in~\eqref{e:Lred}, to order $\sqrt{\lambda}$ as 
\begin{equation}
    \begin{aligned}
        L_{11}=&r^2(L_{11}^1+L_{11}^2 \sqrt{\lambda}) +\cO\left(|r|^3+|\lambda|\right),\\
        L_{11}^1=&\frac{-1+i}{2}\omega_2 - \frac{i+1}{\sqrt{2}} \sqrt{\alpha}\mu_2 -  \beta^2 \Lambda_0^{-1}-\frac{3\gamma}{2} - \frac{\beta^2}{\alpha}- \beta^2 \Lambda_2^{-1},\\
         L_{11}^2=& -\sqrt{2}\frac{\beta^2}{\alpha},\\
        L_{12}=&r^2(L_{21}^1r^2+L_{21}^2 \sqrt{\lambda}) +\cO\left(|r|^3+|\lambda|\right),\\
         L_{12}^1=&-\frac{\beta^2}{2}\Lambda_2^{-1} - \frac{3\gamma}{4} - \frac{\beta^2}{\alpha},\\
         L_{12}^2=&-\sqrt{2}\frac{\beta^2}{\alpha} =L_{11}^2,\\
         L_{22}^j=&\overline{L_{11}^j},\quad j\in\{1,2\},\\
         L_{21}^j=&\overline{L_{12}^j},\quad j\in\{1,2\}.
    \end{aligned}
\end{equation}
Evaluating yet more explicitly by inserting the expressions for $\mu_2$, $\omega_2$, and $\Lambda_2$, adding dependence at $\cO(\lambda)$ from \eqref{e:lam-terms}, and simplifying, we find 
\begin{equation}\label{e:L11L12}
    \mathcal{L}(\lambda;r)=\left(\begin{array}{cc}
    ar^2+b\lambda+d\sqrt{\lambda}r^2&cr^2+d\sqrt{\lambda}r^2\\
    \bar{c}r^2+\bar{d}\sqrt{\lambda}r^2& \bar{a}r^2+\bar{b}\lambda +\bar{d}\sqrt{\lambda} r^2
    \end{array}
    \right)
    \left(\begin{array}{c}
    A\\
    \bar{A}
    \end{array}\right)+\cO\left(|r\lambda|+|r|^3+|\lambda^{3/2}|\right)=0,
\end{equation}
with 
\begin{equation}
    \begin{aligned}
        a&=c=-\frac{3 \left((1+2 i)-(1+i) \sqrt{2}\right)
   \alpha  \gamma +\left((6+8 i)-(4+4 i)
   \sqrt{2}\right) \beta ^2}{\left((4+8 i)-(4+4
   i) \sqrt{2}\right) \alpha },\\
  b&= \frac{1+i}{2} \alpha,\\
  d&=\frac{\left((-2-2 i)+(1+2 i) \sqrt{2}\right)
   \beta ^2}{\left((-1-2 i)+(1+i)
   \sqrt{2}\right) \alpha }.
    \end{aligned}
\end{equation}
We can now proceed to find eigenvalues by computing 
\begin{equation}\label{e:det}
\mathcal{D}(\lambda;r)=\mathrm{det}\,\mathcal{L}(\lambda;r)=|b|^2\lambda^2+2\Re(a\bar{b})\lambda r^2 +\cO(|\lambda|^{5/2}+|\lambda|^2|r|+|\lambda|^{3/2}r^2+|\lambda||r|^3+|r|^5).
\end{equation}
In fact $\mathcal{D}$ is analytic when considered as a function of $\rho=\sqrt{\lambda}$ and $r$.

Since the $s$-derivative of the solution always belongs to the kernel, we have $\mathcal{D}(0;r^2)=0$. The Newton polygon gives that all 4 solutions near the origin scale as $\sqrt{\lambda}=\rho=\rho_1 r$. The absence of $\rho$-terms at the origin, due to the cancellation of $\sqrt{\lambda}$-terms in the determinant, then implies that $\rho=0$ is a double root. The other two roots $\rho=\pm\sqrt{\lambda}$ (only $\Re\rho>0$ corresponds to an eigenvalue) are given through 
\[
\lambda=-\frac{2\Re(a\bar{b})}{|b|^2}r^2+\cO(r^3).
\]
Evaluating gives  
\[
-\frac{2\Re(a\bar{b})}{|b|^2}=\frac{(9-9 i) \left(2 \sqrt{2}-3\right) \alpha  \gamma +\left((-34+42 i)+(24-28 i)
   \sqrt{2}\right) \beta ^2}{24 \sqrt{2}-36},
\]
which one verifies is negative precisely when $\mu_2$ is positive and positive precisely when $\mu_2$ is negative.
\end{proof}
\begin{remark}[Role of essential spectrum in reduced stability]
    The proof is slightly awkward, exhibiting the correspondence between supercritical versus subcritical branching and stability versus instability of the bifurcating branches through an independent direct calculation that exhibits the equivalence. It would clearly be interesting to see this equivalence more directly in some reduced dynamical, rather than Lyapunov-Schmidt reduced, description. We emphasize however that the technical obstruction to a reduction due to continuous spectrum at the origin appears in the proof through the presence of $\sqrt{\lambda}$-terms in the reduced eigenvalue problem --- which then somewhat miraculously cancel. 
\end{remark}

\begin{remark}[Stability and instability with degradation]\label{r:posd}
    In the case $\sigma>0$, the Hopf bifurcation analyzed here can be reduced to a two-dimensional center-manifold due to the spectral gap between the Hopf eigenvalues $\pm i\omega_\ast$ on the imaginary axis and the remainder of the spectrum, in particular the essential spectrum which consists of $(-\infty,-\sigma^2]$. Exchange of stability, that is, stability of the periodic orbit in the supercritical and instability in the subcritical case, are then classical consequences of the analysis of the cubic normal form.
\end{remark}

\section{Numerical asymptotics and homoclinic limits at finite amplitude} \label{sssec:numericalasympt}

We explore solutions to the boundary-value problem numerically.  We focus on the system~\eqref{eq:sys_cubicorder_degrad} without degradation, $\sigma=0$,  and trivial solution $u\equiv 0$ as discussed in Corollary~\ref{cor:m_w_cubicorder}. 

\paragraph{Setup.} Specifically, we compute solutions to 
\begin{equation}\label{e:bvpnum}
    G(u^-) := -D(\omega) u^- -\alpha u^- +\beta (u^-)^2+\gamma (u^-)^3 + (\mu+\sqrt{2\alpha})\; D(\omega)^{1/2} u^- = 0,
\end{equation}
for a $2\pi$-periodic function $u^-=u^-(s)$. 

Since solutions come in one-parameter families due to the shift $u^-(\cdot+s_0)$, we supplement~\eqref{e:bvpnum} with a phase condition, which we chose as 
\begin{equation}\label{e:phasecond}
\int \sin(s) u^-(s) ds =0. 
\end{equation}
This phase condition fixes the time shift as long as the first Fourier mode of $u^-$ does not vanish.
Numerically, we discretized $u^-$ on a grid of 2048 points and evaluated $D(\omega)$ as well as $D(\omega,\sigma)^{1/2}$ using Fast Fourier Transform. We then solved~\eqref{e:bvpnum} together with~\eqref{e:phasecond} for the variables $u^-$ and $\omega$ using a Newton method starting with an initial guess from Theorem~\ref{theorem:hopf} with small $r$ and expansion for $\mu$ and $\omega$ as in Theorem~\ref{cor:m_w_cubicorder}. From two nearby such initial guesses, we then pursue a secant continuation in the parameter $\mu$ to both explore the quality of our asymptotic prediction near $\mu\sim 0$ and the fate of solution branches away from the bifurcation point $(u^-,\mu)=(0,0)$.

\paragraph{Results.}

We noted in~\eqref{e:gamcrit} that $\mu_2$ changes sign when crossing the  curve $\gamma = (4\sqrt{2}-6) \beta^2 / 9$  in parameter space, depicted in Figure~\ref{fig:betagammacrit}. Below the curve, the correction coefficient $\mu_2$ is positive, the bifurcation is supercritical;  it is negative above the curve, corresponding to a subcritical bifurcation. 
\begin{figure}
    \centering
    \includegraphics[width=0.6\textwidth]{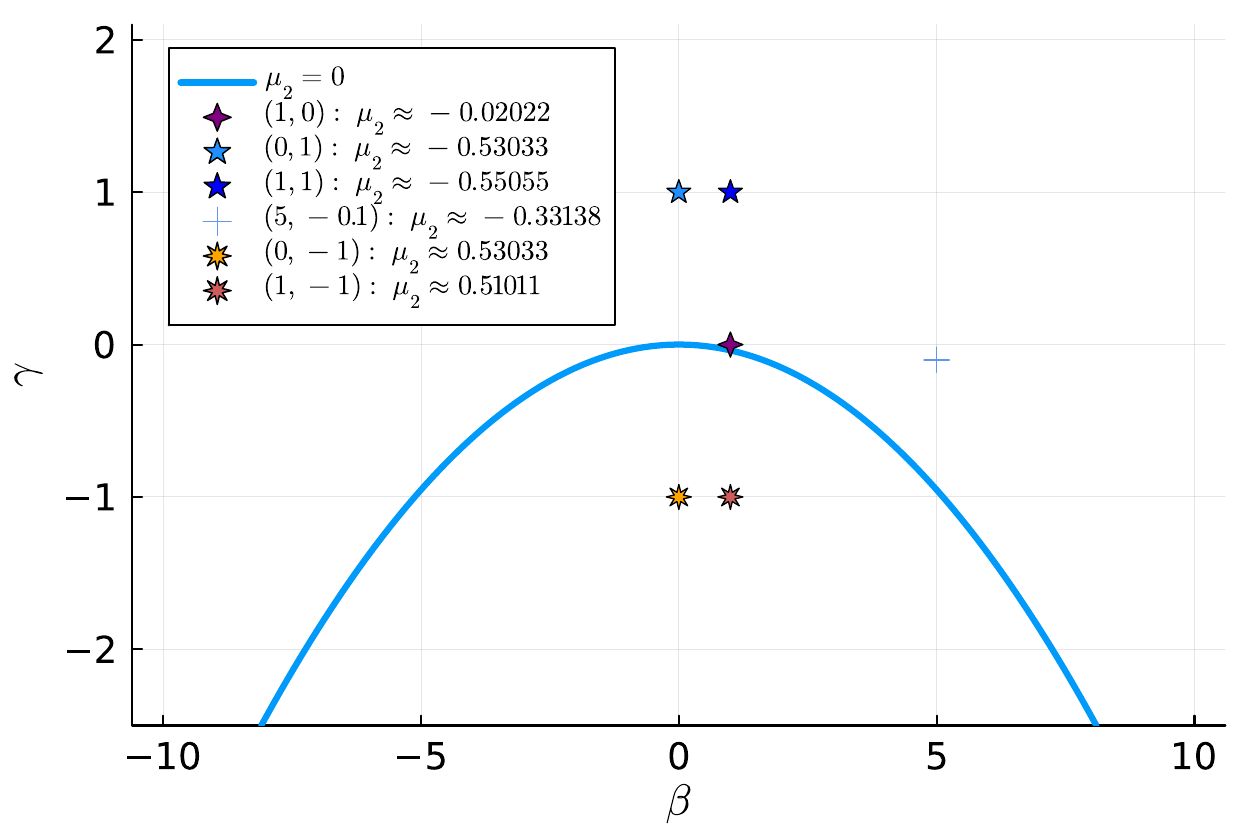}
    \caption{The critical curve $\gamma = (\frac{4\sqrt{2}}{9}-\frac{2}{3}) \beta^2 $ from~\eqref{e:gamcrit} for $\alpha=1$, separating supercritical (below) and subcritical branching is shown together with some sample values of $\mu_2 = 0$ from Corollary~\ref{cor:m_w_cubicorder}.
    Bifurcation diagrams at the sample points are shown in Figures~\ref{fig:hopfcurve_abovecrit_negor0_g},~\ref{fig:hopfcurve_g1}, and~\ref{fig:hopfcurve_g-1}, below.}
    \label{fig:betagammacrit}
\end{figure}
We show results of numerical continuation in Figures~\ref{fig:hopfcurve_abovecrit_negor0_g}--\ref{fig:homomega}. Plotted is the amplitude of the first Fourier mode $r$ against the parameter $\mu$ throughout, according to parameter values identified in Figure~\ref{fig:betagammacrit}. Shown in the graphs are both computed continuation branches (blue) and asymptotic predictions (magenta), with good agreement near the bifurcation point (see insets). As expected, the agreement is lacking when $\mu_2$ is very small, that is, near the transition from sub- to supercritical bifurcation, when $\cO(r^4)$-terms in the expansion for $\mu$ are dominant; see in particular bottom panel of Figure~\ref{fig:hopfcurve_abovecrit_negor0_g}. Insets in all figures show enlarged regions near the bifurcation point and sample profiles of mostly sinusoidal solutions.

\begin{figure}
    \centering
    \begin{subfigure}{0.8\textwidth}
    \includegraphics[width=1\textwidth]{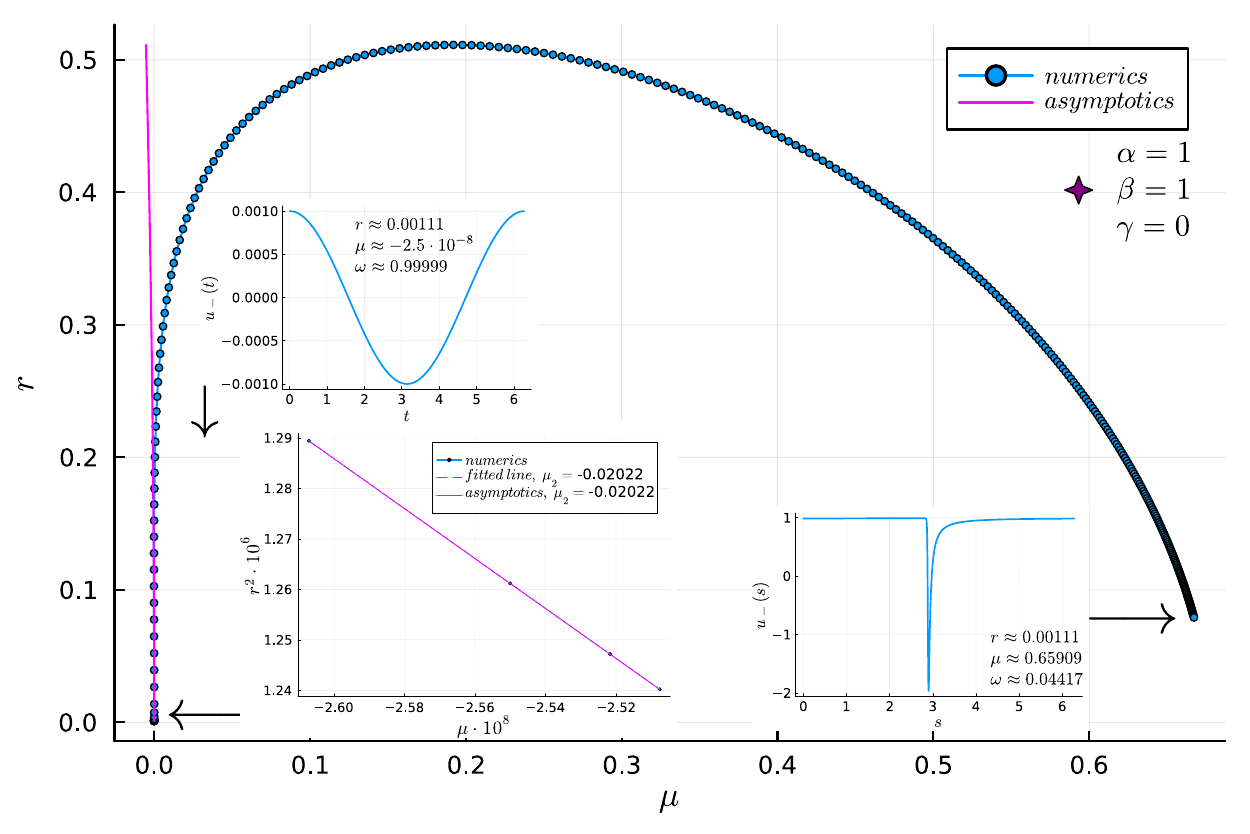}
    \end{subfigure}
    \begin{subfigure}{0.8\textwidth}
    \includegraphics[width=1\textwidth]{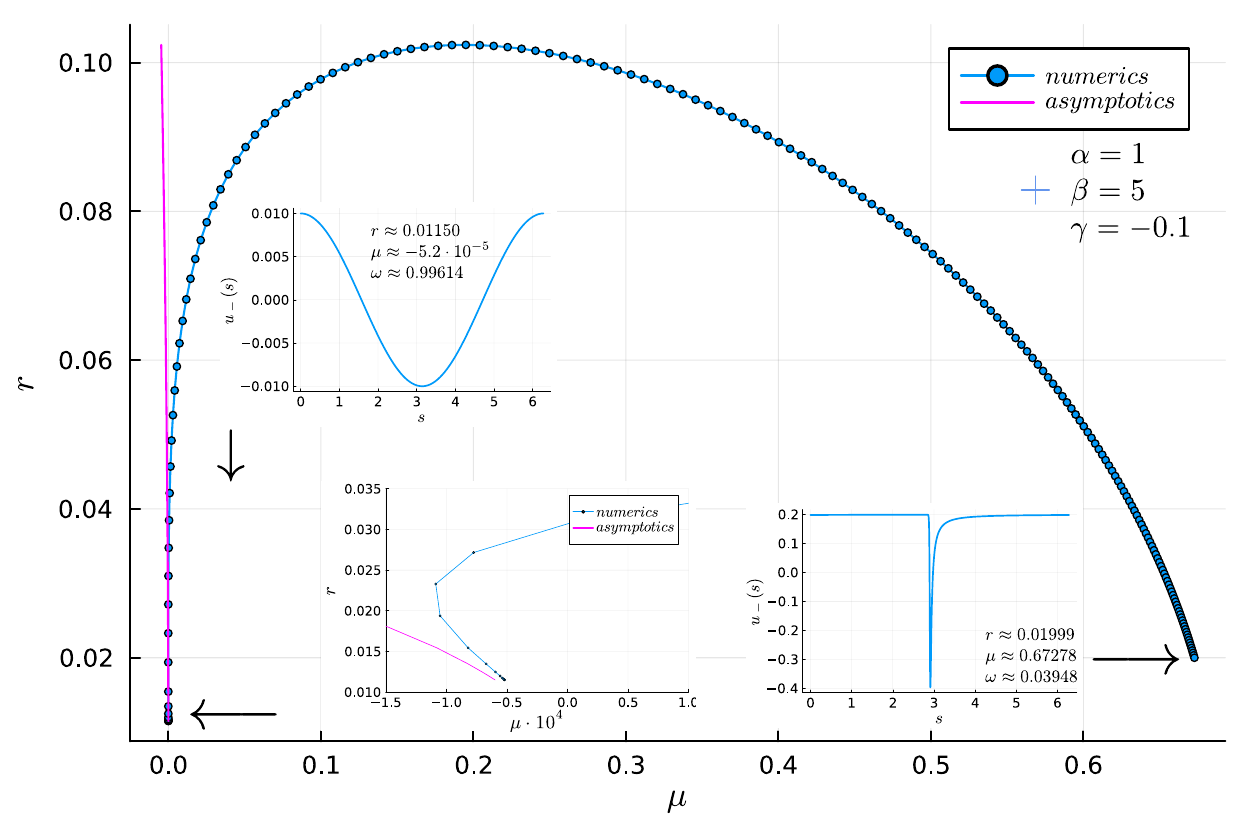}
    \end{subfigure}
    \caption{Bifurcation diagram for~\eqref{eq:sys_cubicorder_degrad} with $\sigma=0$ as discussed in Corollary~\eqref{cor:m_w_cubicorder}, with parameters in the weakly subcritical regime, close to the transition to supercriticality.  Top:  $\alpha=1$, $\beta=1$, and $\gamma=0$ (purple 4-star in Figure~\ref{fig:betagammacrit}). Bottom:   $\alpha=1$, $\beta=5$, and $\gamma=-0.1$ (blue $+$ in Figure~\ref{fig:betagammacrit}). Both diagrams exhibit hysteresis with a slightly subcritical branch turning towards supercritical parameter values. Insets show sample plots of solutions and enlarged diagrams near onset. Asymptotics are shown in magenta for comparison. Note the limit on a homoclinic profile with slowly decaying tails far from the bifurcation point.}
    \label{fig:hopfcurve_abovecrit_negor0_g}
\end{figure}

\begin{figure}
    \centering
    \begin{subfigure}{0.8\textwidth}
    \includegraphics[width=1\textwidth]{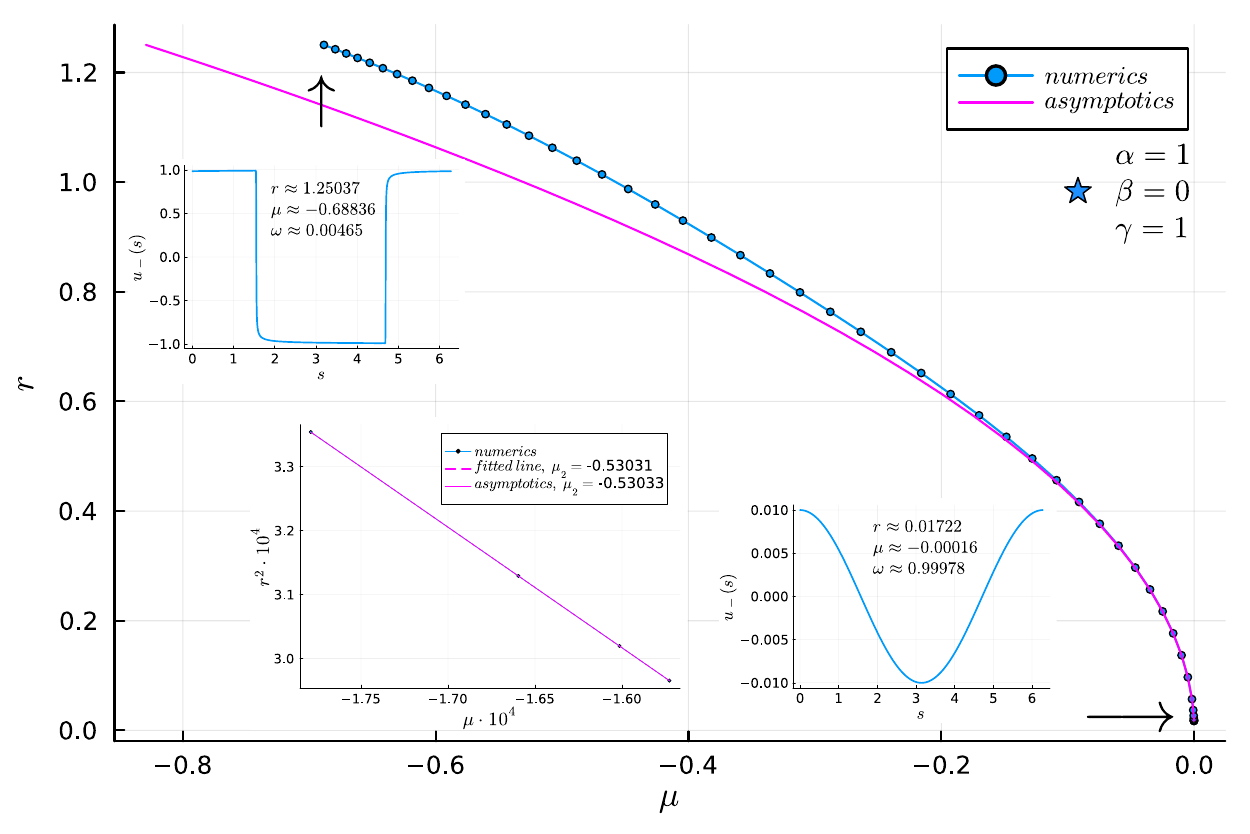}
    \end{subfigure}
    \begin{subfigure}{0.8\textwidth}
    \includegraphics[width=1\textwidth]{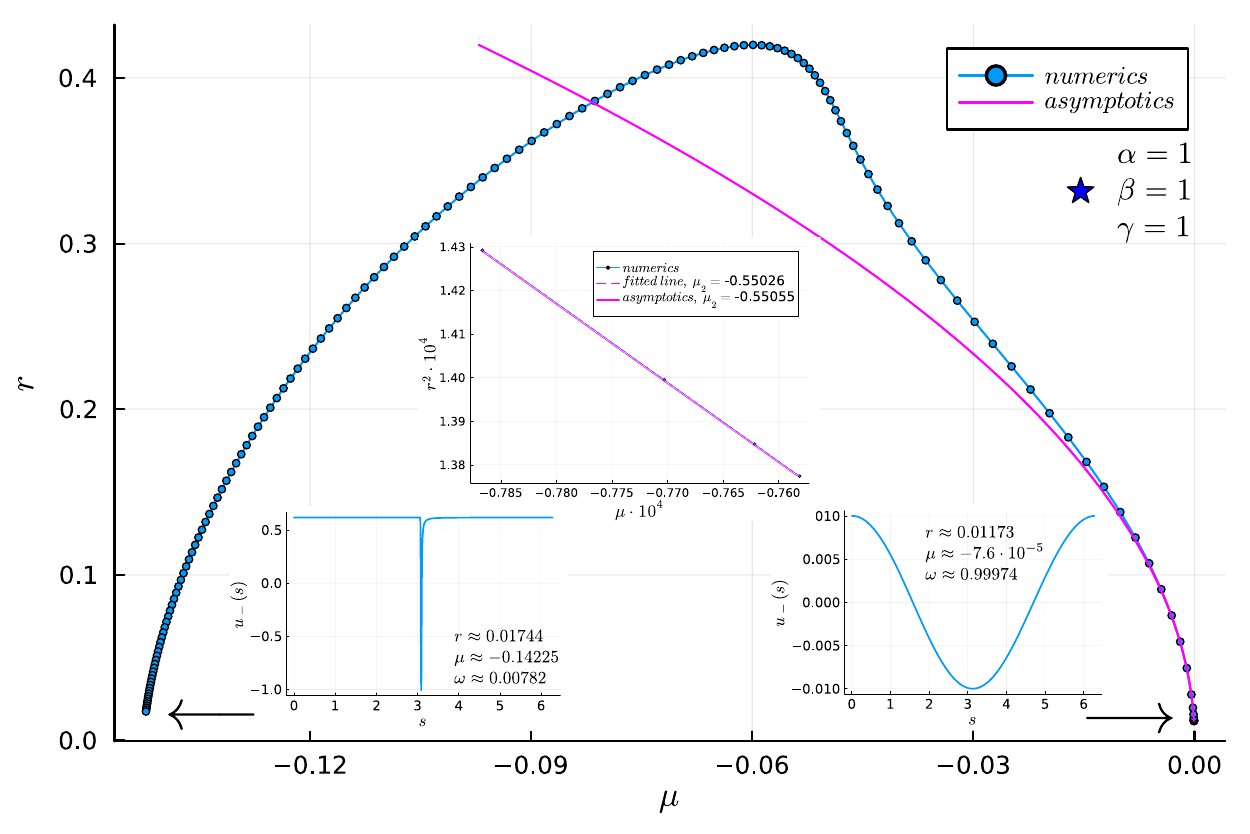}
    \end{subfigure}
    \caption{Bifurcation diagram for~\eqref{eq:sys_cubicorder_degrad} with $\sigma=0$ as discussed in Corollary~\eqref{cor:m_w_cubicorder}, with parameters in the subcritical regime. Top: $\alpha=1$, $\beta=0$, and $\gamma=1$ (light blue 5-star in Figure~\ref{fig:betagammacrit}). Bottom:  $\alpha=1$, $\beta=1$, and $\gamma=1$ (dark blue 5-star in Figure~\ref{fig:betagammacrit}). Insets show sample plots of solutions and enlarged diagrams near onset. Asymptotics are shown in magenta for comparison. Note the limit on a heteroclinic loop profile (top) and a homoclinic profile (bottom), both with long tails far from the bifurcation point. }
    \label{fig:hopfcurve_g1}
\end{figure}

\begin{figure}
    \centering
    \begin{subfigure}{0.8\textwidth}
    \includegraphics[width=1\textwidth]{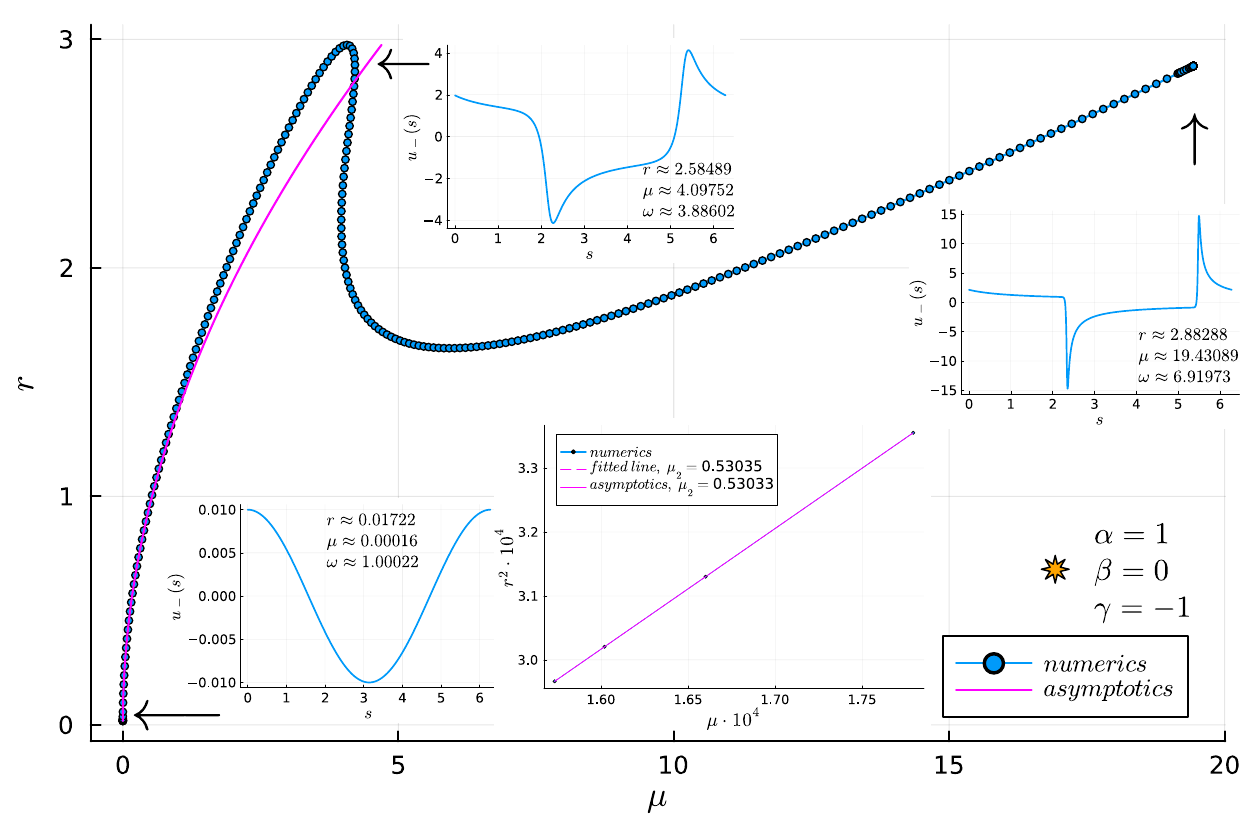}
    \label{fig:hopfcurve_a1b0g-1}
    \end{subfigure}
    \begin{subfigure}{0.8\textwidth}
    \includegraphics[width=1\textwidth]{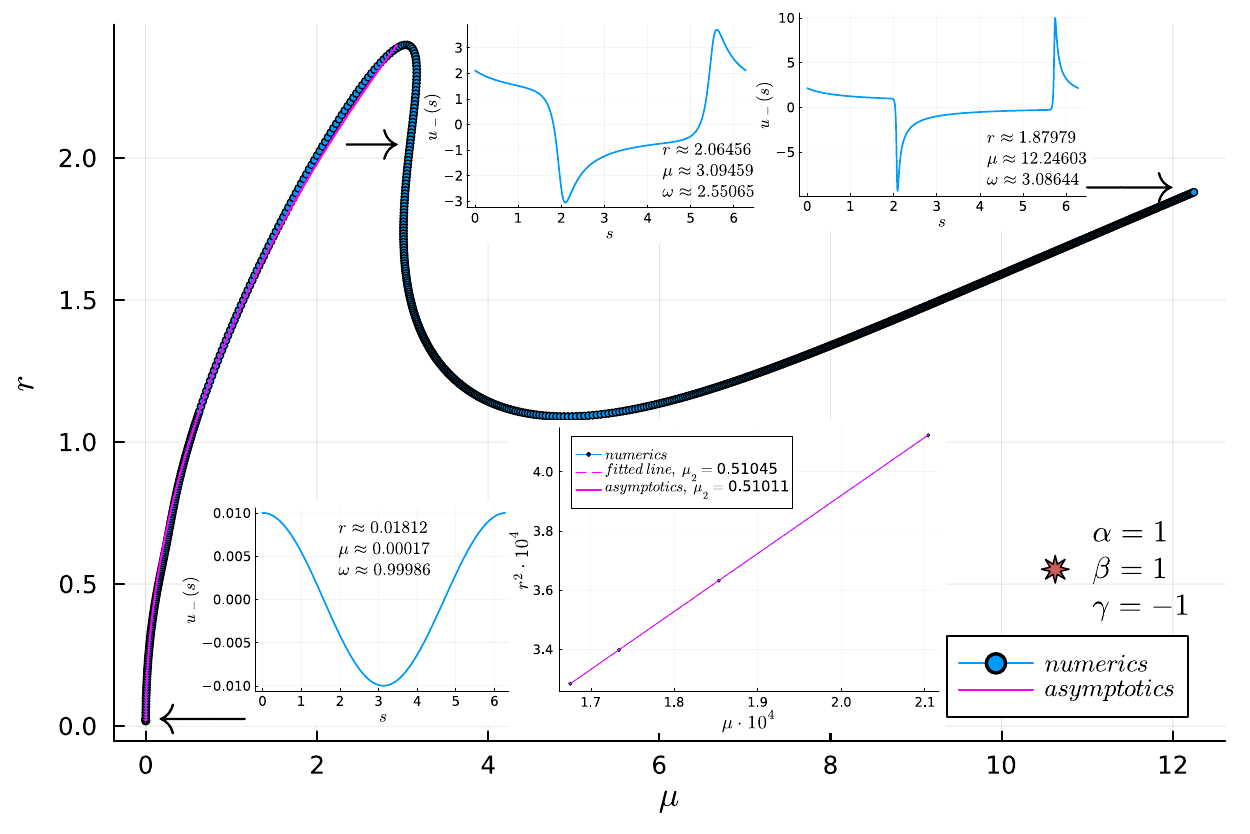}
    \label{fig:hopfcurve_a1b1g-1}
    \end{subfigure}
    \caption{Bifurcation diagram for~\eqref{eq:sys_cubicorder_degrad} with $\sigma=0$ as discussed in Corollary~\eqref{cor:m_w_cubicorder}, with parameters in the subcritical regime.  Top: for $\alpha=1$, $\beta=0$, and $\gamma=-1$ (yellow 8-star in Figure~\ref{fig:betagammacrit}). Bottom: for $\alpha=1$, $\beta=1$, and $\gamma=-1$ (red 8-star in Figure~\ref{fig:betagammacrit}). Insets show sample plots of solutions and enlarged diagrams near onset.
    Asymptotics are shown in magenta for comparison. Note that the amplitude of solutions diverges and solutions approach a double-homoclinic loop in both cases, respecting the symmetry $u^-(\cdot)=-u^-(\cdot+\pi)$ of the odd nonlinearity (where quadratic terms are negligible at large amplitudes).
    }
    \label{fig:hopfcurve_g-1}
\end{figure}
Figure~\ref{fig:hopfcurve_abovecrit_negor0_g} shows weakly subcritical branches that bend toward positive $\mu$ with expected stable oscillations. Figure~\ref{fig:hopfcurve_g1} shows subcritical branching with negative positive $\gamma$ and negative $\mu_2$, both with $\beta=0$ and $\beta\neq 0$. The case $\beta=0$ demonstrates, in particular away from the bifurcation point, the spatio-temporal symmetry $u^-(\cdot+\pi)=-u^-(\cdot)$, that is, a sign change after half the period.
Finally, Figure~\ref{fig:hopfcurve_g-1} shows the corresponding situation with $\gamma=-1$ and supercritical bifurcations leading to stable periodic solutions. We also compared expansions for the frequency $\omega$ in Figure~\ref{fig:homomega} with good agreement.

Interestingly, bifurcation diagrams are not monotone, exhibiting what we expect to be multistability, as visible both in the weakly subcritical bifurcations of Figure~\ref{fig:hopfcurve_abovecrit_negor0_g} and the non-monotone bifurcation curves in Figure~\ref{fig:hopfcurve_g-1}. 
Generally, we find good agreement of the quadratic expansion $\mu_2$; compare insets in Figures~\ref{fig:hopfcurve_abovecrit_negor0_g},~\ref{fig:hopfcurve_g1}, and~\ref{fig:hopfcurve_g-1}.

Lastly, numerical continuation eventually breaks down for various reasons. In the supercritical cases $\gamma=-1$, Figure~\ref{fig:hopfcurve_g-1}, we were able to continue for large $\mu$-values with correspondingly large amplitudes of oscillations, which we suspect exist for all $\mu>0$ with amplitude converging to infinity as $\mu\to\infty$. In all other cases, $\gamma\leq 0$ or $\gamma\ll \beta$, oscillations terminate on a homoclinic orbit or a heteroclinic loop, visible in the sharp spikes or transition layers of the inset plots, as well as in the plots of frequencies $\omega$ which converge to 0 as the end of the bifurcation path is approached. Such homoclinic limits are of course common in delay-differential equations, reinforcing the interpretation of the buffering effect of the diffusive bulk as a delay term in the equation.

\begin{figure}
\centering
\begin{subfigure}{0.49\textwidth}
    \includegraphics[width=8cm, height=3.5cm]{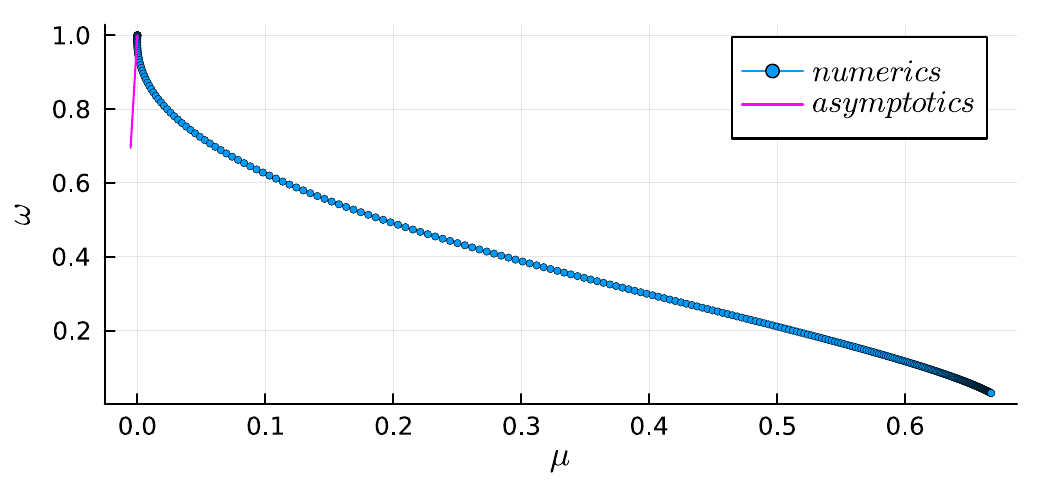}
    \caption{Parameters: $\alpha = 1$, $\beta = 1$, and $\gamma = 0$, as for the purple 4-star in Figure~\ref{fig:betagammacrit} and Figure~\ref{fig:hopfcurve_abovecrit_negor0_g} (top).}
    \label{fig:homa1b1g0}
\end{subfigure}
\hfill
\begin{subfigure}{0.49\textwidth}
    \includegraphics[width=8cm, height=3.5cm]{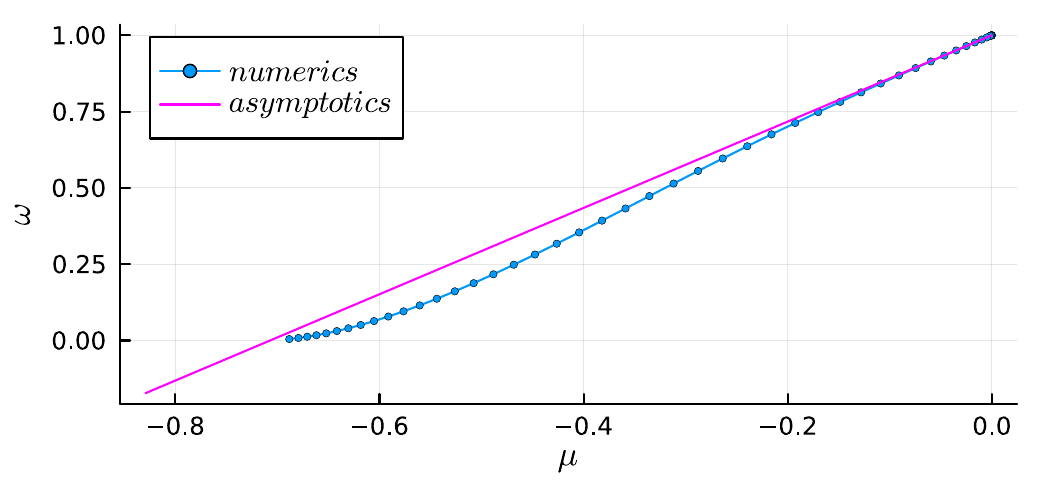}
    \caption{Parameters: $\alpha = 1$, $\beta = 0$, and $\gamma = 1$, as for the light blue 5-star in Figure~\ref{fig:betagammacrit} and as for Figure~\ref{fig:hopfcurve_g1} (top).}
    \label{fig:homa1b0g1}
\end{subfigure}
\begin{subfigure}{0.49\textwidth}
    \includegraphics[width=8cm, height=3.5cm]{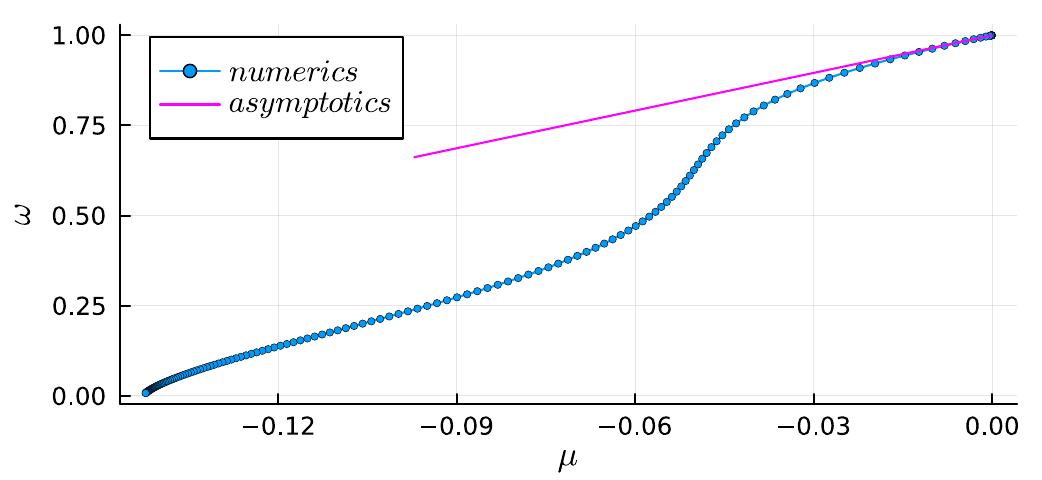}
    \caption{Parameters: $\alpha = 1$, $\beta = 1$, and $\gamma = 1$, as for the dark blue 5-star in Figure~\ref{fig:betagammacrit} and as for Figure~\ref{fig:hopfcurve_g1} (bottom).}
    \label{fig:homa1b1g1}
\end{subfigure}
\hfill
\begin{subfigure}{0.49\textwidth}
    \includegraphics[width=8cm, height=3.5cm]{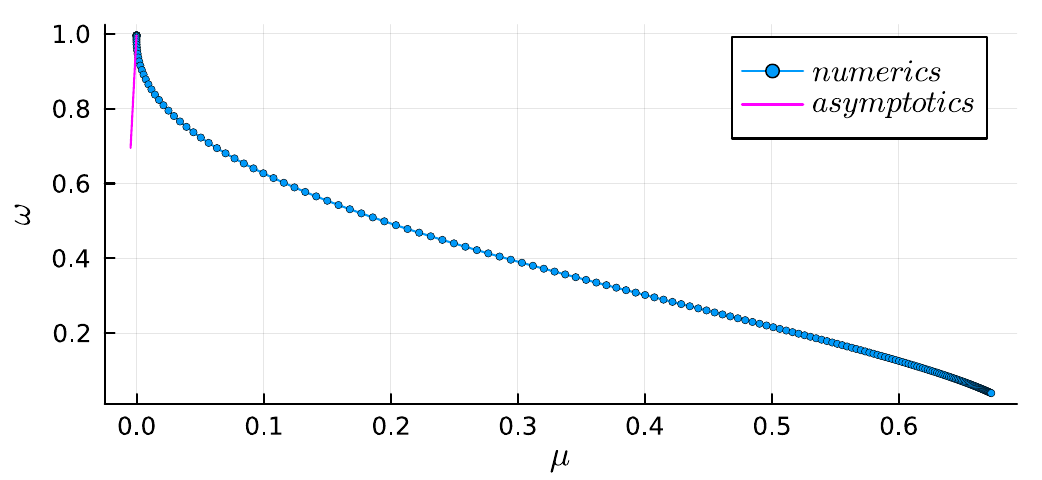}
    \caption{Parameters: $\alpha = 1$, $\beta = 5$, and $\gamma = -0.1$, as for the blue $+$ in Figure~\ref{fig:betagammacrit} and as for Figure~\ref{fig:hopfcurve_abovecrit_negor0_g} (bottom).}
    \label{fig:homa1b5g-0p1}
\end{subfigure}
\caption{Comparison of frequencies with asymptotics (magenta) and homoclinic  or heteroclinic loop (both zero frequency) limits shown in plots of $\omega$ versus bifurcation parameter $\mu$.  See captions in  subfigures for parameters.}
\label{fig:homomega}
\end{figure}

\section{Discussion} \label{sec:discussion}

We discuss both mathematical aspects and extensions of the results herein and other potential applications.

\paragraph{Large bounded domains.}

It would be interesting to study potential interactions of Hopf bifurcations happening in a symmetric system of size $2L\gg 1$ at both boundaries, that is,
\[
\partial_tu=\partial_{xx}u,\ |x|<L, \qquad u|_{x=\pm L}=u^\pm,\qquad \frac{d}{dt}u^\pm=f(u^\pm,\partial_n u|_{x=\pm L}).
\]
For $L\gg 1$, one finds two pairs of purely imaginary eigenvalues crossing almost simultaneously, generating localized oscillations at both boundaries that interact weakly through the domain. Interestingly, the presence of two Hopf bifurcations suggests that the dynamics here could be yet more complicated, with the potential presence of invariant tori. The exponentially weak interaction of the two bifurcations appears similar to the interaction in the Nonlinear Schr\"odinger equation with double wells studied in \cite{KKSW}. 

\paragraph{Systems of equations.}
The analysis we carry out clearly can be adapted to systems of equations, with suitable assumptions where one takes determinants for constructing $\dm$ and assumes similar simple spectral crossing.  An interesting question then is to understand, even on the linearized level, in which way the boundary coupling specifically can enable oscillations that would not be present without it, and whether this mechanism would require the type of active sensing, that is, positive feedback from $\partial_n u$ rather than the typical negative feedback from diffusive balance; see also the discussion of Min-protein oscillations below.

\paragraph{Higher space dimension.}
More ambitious are generalizations to higher space-dimensions, where, say, reactions occur on the boundary of a bounded domain $\Omega$ and are coupled to a diffusive field in $\R^N\setminus \Omega$, even in simple geometries such as $\Omega=B_R$, a ball of radius $R$. In addition to technical complications due to the interaction with essential spectrum that is now represented by infinitely many spherical harmonics on $\partial\Omega$, one would always find a substantially more involved Hopf bifurcation problem with symmetry $\mathrm{O}(n)$; see for instance~\cite{fiedler,ioossrossi}.

\paragraph{Homoclinic limits and  large delays.}

The limiting points in the numerical global  bifurcation diagrams that we computed are strongly reminiscent of slowly oscillating solutions in delay equations with large delay, or equivalently, singularly perturbed delay equations, 
\begin{equation} \label{eq:delayDE}
    \varepsilon \dot{x}(t) = -x(t) + f(x(t-1)).
\end{equation}
Formally, one may wish to set $\varepsilon=0$ and find a simple one-dimensional interval map $x_{t}=f(x_{t-1})$ with potential nontrivial periodic orbits and even chaotic dynamics. 
Existence of such ``slowly oscillating solutions'' has been established in several scenarios; see for instance~\cite{mallet1986bifurcation} where a global continuation technique
based on degree theory was deployed to continue the branch of periodic solutions. Asymptotically as $\varepsilon \to 0$, the periodic solutions have square wave shape, with regular convergence if $f$ is monotone and often non-uniform convergence if not, ``reminiscent of the Gibbs phenomenon of Fourier series'', and reminiscent of the periodic solutions we found in Figure~\ref{fig:hopfcurve_g1} (top panel). We also refer to the survey~\cite{nussbaum2003singperturb}  for further  analytical results on slowly oscillating periodic solutions to~\eqref{eq:delayDE}.

Closer to our setting would be delay dependence in the linear term rather than the nonlinearity, but few results appear to tackle that situation directly. Linear dependence on the flux term, and therefore linear delay terms, arises of course  more commonly for systems with linear bulk diffusion and otherwise nonlinear kinetics (without obvious delay) on the boundary; see for instance~\cite{gomez2007self, levine2005membrane, gomez2021pattern,roeger}. In this context, we also mention~\cite{pelz2025synchronized} where a class of such bulk diffusion, compartmental-reaction systems are analyzed using an extension of strong localized perturbation theory, asymptotically reducing to an integro-differential equation with a linear coupling term that represents the diffusive coupling between all well-separated compartments embedded in the linear diffusion field. The results demonstrate the emergence of novel synchronization behavior and symmetry-breaking in such systems, which describe  systems of biological cells more realistically incorporating  their natural compartmentalization in the model. 

\paragraph{Strain, dislocations, and patterns.} Homoclinic phenomena strikingly similar to the termination points observed here have also been found in models for dislocations in  elastic materials, notably in the Weertman equation~\cite{josien2018mathematical,josien2018fourier} and in models for striped phases in soft matter~\cite[\S5]{chen2021strain}.

\paragraph{Planarians and canards.} Square waves as periodic solutions with large period and relaxation oscillation type behavior were also observed in the model for planarian regeneration, where the Hopf bifurcation in fact has frequency $\varepsilon\ll 1$~\cite{scheel2021signaling}. In direct simulations, there appears to be a very steep branch of periodic solutions, similar to the transition to relaxation oscillations in the van der Pol oscillator through a Canard unfolding~\cite{diener,KS1}. Singular perturbation theory in this geometric spirit has not been attempted in systems of the type we study here; see however~\cite{,hupkessand,fayescheel,nonlocdich,tao} for technical tools and results pointing in this direction.

\paragraph{Min-protein oscillations on membranes.} The effect of delay through buffering in a diffusive bulk region has also been identified as a key mechanism for the dynamics of Min-proteins on cell membranes~\cite{brauns2021bulk}. The authors investigate intracellular Min
protein dynamics in E. coli, which control E. coli cell division,  focussing  on only the dynamics of the
concentrations of MinD protein and of MinE protein, the cell membrane
(lipid bilayer) bounding the cytosol (bulk; intracellular), and the
chemical species ADP and ATP.

Oscillations in the concentrations of the Min proteins are reported along
the anterior-posterior axis of the cylinder-shaped cells in vivo. In
vitro, they are additionally reported along the dorsal-ventral axis
(denoted by the vertical direction below; cytosol height).
Spatiotemporal patterns along the vertical direction are not present in
vivo, and hence a much wider range of spatiotemporal patterns can be
observed in vitro. Oscillations as reported in our work here relate to dynamics along the vertical axis of a long E. coli cell (without lateral effects) in the regime of \emph{large bulk heights}, that is, almost decoupled upper and lower membrane dynamics. The results in~\cite{brauns2021bulk} suggest traveling-wave type propagation of oscillations from the boundary into the bulk, similar to the evanescent waves that we find in our analysis. The authors stress the importance of spatial
gradients in the vertical direction, which our system includes in the
boundary kinetics.  The linear stability analysis there exhibits that the coupling to the bulk with induced  buffering effects is crucial to the emergence of oscillations, an effect which we isolated in the overly simple scalar equation~\eqref{eq:sys}.

\paragraph{Other applications.} More generally speaking, our results and dynamic boundary conditions are related to strong variations in diffusivity or additional strong fluxes near a boundary compartment, for instance a biological membrane. An example are  presynaptic neurons, where transporters can modify the effective diffusivity of an extracellularly diffusing chemical species, such as serotonin transporters (SERT) that are of key importance in the serotonin reuptake into presynaptic neurons~\cite{coleman2016x, hasenhuetl2016electrogenic, henke2017toward, yang2021illumination}. 
More generally, modified transport properties near the boundary can be generated by mechanical forces similar to the Marangoni effect (equilibrating dynamics due to differences in surface tension of two meeting fluids; e.g., causing liquid to flow up the inner surface of a glass filled with a drink, ``tears of wine'';~\cite{thomson1855xlii, marangoni1865sull, gibbs1878equilibrium}). Such forces have been found to be relevant in the self-organization of cells to form the body axes such as the primary head-to-tail (anterior-posterior) axis in embryos~\cite{gsell2025marangoni}. A range of further studies have shown that molecules critical for morphogenesis, so-called morphogens, generate intercellular mechanical forces that create the geometry and internal kinetics of the organism~\cite{yang2023morphogens, cheikh2025scaling}. Lastly, we also mention that there could also be significant electrochemical forces close to a cell membrane~\cite{spitzer2005electrochemical, israelachvili2013intersection}. 

\paragraph{Acknowledgments.} A.S. gratefully acknowledges partial support through NSF grant  DMS-2506837.

\bibliographystyle{abbrv}
\bibliography{refs}

\end{document}